\documentclass[11pt,english,reqno]{amsart}
\usepackage[T1]{fontenc}
\usepackage[latin9]{inputenc}
\usepackage{amstext}
\usepackage{amsthm}

\makeatletter
\usepackage{amsmath,amsfonts,amssymb,amsthm,epsfig}

\usepackage{color}
\usepackage[final,allcolors=blue,colorlinks=true]{ }

\def\R{\mathbb R}

\def\de{\delta}
\def\ep{\varepsilon}

\def\var{\varphi}
\def\na{\nabla}
\def\Om{\Omega}  
\def\De{\Delta}      

\def\L{\mathcal L}

\def\wq{\infty}
\def\pa{\partial}
\def\loc{\text{\rm loc}}

\newcommand{\wto}{\rightharpoonup}                

\usepackage[numbers,sort&compress]{natbib}

\numberwithin{equation}{section}
\newtheorem{theorem}{Theorem}[section]

\newtheorem{lemma}[theorem]{Lemma}

\newtheorem{proposition}[theorem]{Proposition}
\newtheorem{remark}[theorem]{Remark}

\theoremstyle{definition}

\makeatother

\usepackage{babel}
\begin{document}
\title[A logarithmic Schr\"{o}dinger equations]{Positive multi-peak solutions for a logarithmic Schr\"{o}dinger equation}

 \author[P. Luo and Y. Niu]{Peng Luo and Yahui Niu}

\address[Peng Luo]{ School of Mathematics and Statistics and Hubei Key Laboratory of Mathematical Sciences,  Central China Normal University, Wuhan,  430079, P.R. China}
\email[]{pluo@mail.ccnu.edu.cn}
\address[Yahui Niu]{School of Mathematics and Statistics and Hubei Key Laboratory of Mathematical Sciences, Central China Normal University,  Wuhan,  430079, P.R. China}
\email[]{yahuniu@163.com}

\begin{abstract}
 In this paper, we consider the logarithmic Schr\"{o}dinger equation
 \begin{eqnarray*} -\ep^2\Delta u+V(x)u=u\log u^{2},\,\,\,u>0,  & \text{in }\mathbb{R}^{N}, \end{eqnarray*} where $N\geq3$, $\varepsilon>0$ is a small parameter. Under some assumptions on $V(x)$, we show the
   existence of positive multi-peak solutions by Lyapunov-Schmidt reduction. It seems to be the first time to study singularly perturbed logarithmic Schr\"{o}dinger problem by reduction.
   And here using a new norm is the crucial technique to overcome the difficulty caused by the  logarithmic nonlinearity.
   At the same time,
   we consider the local uniqueness of the multi-peak solutions by using a type of local Pohozaev identities.
\end{abstract}


\maketitle

{\small
\smallskip
\newline
}
\bigskip

\section{Introduction and main results}

In this paper, we consider the
following logarithmic Schr\"{o}dinger equations
\begin{eqnarray}
-\ep^2\Delta u+V(x)u=u\log u^{2},\,\, u>0, &\text{in}~\mathbb{R}^{N},\label{eq: Kirchhoff}
\end{eqnarray}
where $\ep>0$ is a parameter, $N\geq3$.

Eq. \eqref{eq: Kirchhoff} is closely related to the time-dependent logarithmic Schr\"{o}dinger equations
\begin{equation}\label{time-dependent equ}
  i\ep\partial_{t}u+\frac{\ep^{2}}{2}\triangle u-V(x)u+u\log u^{2}=0.
\end{equation}
Eq. \eqref{time-dependent equ} was proposed by Bialynicki-Birula and Mycielski \cite{Bialynicki-1976} as a model of nonlinear wave mechanics. This NLS Eq. \eqref{time-dependent equ} has wide applications in quantum optics \cite{Buljan-2003}, nuclear physics \cite{Hefter-1985},
geophysical applications of magma transport \cite{De Martino-2003}, effective quantum and gravity, theory of superfluidity, Bose-Einstein condensation and open quantum systems(see \cite{Zloshchastiev-2010,Zloshchastiev-2011} and the references therein).
For the existence, stability of standing waves and the Cauchy problem in a suitable functional framework about
Eq. \eqref{time-dependent equ}, we can refer to \cite{Ardila-2016,Ardila-2017,Cazenave-1983,Cazenave-1980,Cazenave-1982}.

 We call $u\in H^1(\R^{N})$ a (weak) solution to Eq. \eqref{eq: Kirchhoff} if  it holds that
\[
\ep^{2}\int_{\R^{N}}\nabla u\nabla\psi+\int_{\R^{N}}V(x)u\psi=\int_{\R^{N}}u\psi\log u^2,~\mbox{for any}~\psi\in H^1(\R^{N}).
\]

From a variational point of view, the search of nontrivial solutions to  \eqref{eq: Kirchhoff} can be
formally associated with the study of critical points of the functional on $H^{1}(\R^{N})$ defined by
\[
I_{\varepsilon}(u)
=\frac{\ep^{2}}{2}\int_{\R^{N}}|\nabla u|^{2}+\frac{1}{2}\int_{\R^{N}}(V(x)+1)u^{2}-\frac{1}{2}\int_{\R^{N}} u^{2}\log u^{2},\ \ u\in H^1(\R^{N}).
\]
By using the following standard logarithmic Sobolev inequality (see Theorem 8.14 in \cite{Lieb-2001})
\[
\int_{\R^{N}} u^{2}\log u^{2}\leq\frac{a^{2}}{\pi}\|\nabla u\|^{2}_{2}+\Big(\log\|u\|^{2}_{2}-N(1+\log a)\Big)\|u\|^{2}_{2},\ u\in H^{1}(\R^{N}),\ a>0,
\]
it is easy to see that $\displaystyle\int_{\R^{N}} u^{2}\log u^{2}<+\infty$ for all $u\in H^{1}(\R^{N})$, but there exists $u\in H^{1}(\R^{N})$ such that $\displaystyle\int_{\R^{N}} u^{2}\log u^{2}=-\infty$. For example, if $N=1$, $u$ is a smooth function satisfying
\[
\begin{aligned}
u(x)=\left\{
       \begin{array}{ll}
         \big(\sqrt{x}\log x\big)^{-1}, & x\geq3; \\
         0, & x\leq2.
       \end{array}
     \right.
\end{aligned}
\]
One can verify directly that $u\in H^{1}(\R^{N})$ and $\displaystyle\int_{\R^{N}} u^{2}\log u^{2}=-\infty$.
 Thus, in general, $I_{\varepsilon}(u)$ fails to be finite and $C^{1}$ smooth on $H^{1}(\R^{N})$.

Due to this loss of smoothness,  the classical critical point theory cannot be applied for $I_{\varepsilon}$. In order to study existence of solutions to logarithmic Schr\"{o}dinger equation, several approaches were used so far in the literature as far as we know.
For  problem \eqref{eq: Kirchhoff} with $\varepsilon=1$, Cazenave \cite{Cazenave-1983} worked in a suitable Banach space $W$ endowed with a Luxemburg type norm in order to make the functional $I_{1} : W \rightarrow \R$ well defined and  $C^{1}$ smooth. In recent years, non-smooth critical point theory was applied , such as Squassina and Szulkin \cite{Squassina-2015,Squassina-2017} studied the following logarithmic Schr\"{o}dinger equation
\begin{equation}\label{double potential eq}
  -\Delta u+V(x)u=Q(x)u\log u^{2}, \ \ \text{in}\ \R^{N},
\end{equation}
where $V (x)$ and $Q(x)$ are spatially periodic. They showed the existence of ground state and infinitely many possibly sign-changing solutions, which are geometrically distinct under $\mathbb{Z}^{N}$-action. See also \cite{dAvenia-2014,dAvenia-2015,Ji-2016} for more non-smooth variational framework to logarithmic Schr\"{o}dinger equation. At the same time,
 by using penalization technique, Tanaka and Zhang \cite{Tanaka-2017}  obtained infinitely many multi-bump geometrically distinct positive solutions of (\ref{double potential eq}).
We also refer to  \cite{Guerrero-2010} for the approach of using penalization.
Another interesting work concerning with Eq. \eqref{eq: Kirchhoff} with $\varepsilon=1$ is \cite{Shuai-2019},  by using the constrained minimization method, which avoided using Luxemburg type norm, non-smooth critical point theory
and penalization technique. Here Shuai \cite{Shuai-2019}  proved directly the minimizers of $I_{1}(u)$ on a Nehari
set or a sign changing Nehari set are indeed solutions by direction derivative.

Recently,  problem \eqref{eq: Kirchhoff}  was studied in \cite{Alves-2018} if  $V(x)$ is a continuous function with a global minimum. By using variational method developed by Szulkin in \cite{Szulkin-1986} for functionals which are sum of a ${C}^{1}$ functional with a convex
lower semi-continuous functional, Alves et al in \cite{Alves-2018} proved, for $\ep>0$ small enough, the existence of positive solutions and concentration around of a minimum point of $V(x)$.
Later, Alves and Ji in \cite{Alves-2019} studied the existence of multiple solutions for problem \eqref{eq: Kirchhoff} under the following conditions on potential $V(x)$:

\smallskip

\noindent(I). $V :\R^{N} \rightarrow \R$ is a continuous function such that
\[
\lim_{|x|\rightarrow\infty}V(x)=V_{\infty} ~\mbox{and}~0<V(x) < V _{\infty}~\mbox{for any}~x\in\R^{N}.
\]
\noindent(II). There exist $l$ points $z_{1},\cdots,\ z_{l}$ in  $\R^{N}$  such that
\[
1=V(z_{i})=\min_{x\in\R^{N}}V(x),\ \  \text{for}\ 1\leq i\leq l.
\]
They proved that for $\ep>0$ small enough, the "shape" of the graph of the function $V$ affects the number of nontrivial solutions,  specifically, Eq. \eqref{eq: Kirchhoff} has at least $l$ positive solutions for $\ep$ small enough.

From the above results, we summarize that all existing results on logarithmic Schr\"{o}din-ger equations are obtained by variational methods. In this paper, we intend to study logarithmic Schr\"{o}dinger  equation \eqref{eq: Kirchhoff} by Lyapunov-Schmidt reduction.

\smallskip

\noindent More precisely, we suppose that $V(x)\in C^{1}:\R^{N}\to\R$
satisfies the following conditions:

\smallskip

\noindent $(V_{1}). $ $V(x)\in L^{\wq}(\R^{N})$ and $0<\inf_{\R^{N}}V(x)\le\sup_{\R^{N}}V(x)<\wq$;

\smallskip

\noindent $(V_{2}). $ There exist $k$ points $\xi_{1},\cdots,\xi_k$ such that
$$\nabla V(\xi_{j})=0,~det\Big(\big(\frac{\partial^{2}V(\xi_{j})}{\partial\xi_{j,i}\partial\xi_{j,l}}\big)_{1\leq i,l\leq N}\Big)\neq 0,~\mbox{for any}~ j=1,\cdots,k.$$

\smallskip

Here we also give the definition of $k$-peak solutions of Eq.
\eqref{eq: Kirchhoff} as usual.

\smallskip

\noindent\textbf{Definition A.}\emph{ Let $k\in \mathbb{N}$ and $\xi_{j}\in \mathbb{R}^{N}$ with $j=1,\cdots,k$. We say that $u_{\ep}\in H^{1}(\R^{N})$
is a $k$-peak solution of \eqref{eq: Kirchhoff} concentrated at $\xi_{1},\cdots,\xi_{k}$ if}

\smallskip

\noindent \emph{\textup{(i)} $u_{\ep}$ has $k$ local maximum points $y_{\ep,j}\in\R^{N}$,
$j=1,\ldots,k$, satisfying
\[
y_{\ep,j}\to \xi_{j},~\mbox{as}~\ep\to 0.
\]
\noindent \textup{(ii)} For any given $\tau>0$, there exists $R\gg1$, such that
\begin{eqnarray*}
|u_{\ep}(x)|\le\tau,  & \text{for}~x\in\R^{N}\backslash\displaystyle\bigcup_{j=1}^{k}B_{R\ep}(y_{\ep,j});
\end{eqnarray*}
\noindent \textup{(iii)} There exists $C>0$ such that
\[
\int_{\R^{N}}(\ep^{2}|\na u_{\ep}|^{2}+u_{\ep}^{2})\le C\ep^{N}.
\] }

Our first result concerning on the existence of  $k$-peak solutions to \eqref{eq: Kirchhoff} is as follows.

\begin{theorem} \label{thm: main reuslt-existence}
Assume that $N\geq3$, $(V_{1})$ and $(V_{2})$ holds. Then, Eq.
\eqref{eq: Kirchhoff} has a $k$-peak solution concentrated at $\xi_{1},\cdots,\xi_{k}$ for $\ep>0$ sufficiently small.\end{theorem}

Now we outline the main ideas and difficulties  in the proof of Theorem \ref{thm: main reuslt-existence}.
The basic idea is to use the unique positive solution to the limiting equation of \eqref{eq: Kirchhoff} as a building block to construct solutions for \eqref{eq: Kirchhoff}.
We first reduce the problem to a finite dimensional one by  Lyapunov-Schmidt reduction.
Since the singularity of the nonlinear term $u\log u^{2}$, traditional reduction method (for example refer to \cite{Bahri-1989}) can't be used directly, we make a few modifications.


Here we introduce some notations. Denote
\[
\langle u,v\rangle_{\ep}=\int_{\R^{N}}\Big(\ep^{2}\na u\cdot\na v+(V(x)+1)uv\Big),~H_{\ep}=\big\{u\in H^{1}(\R^{N}):\|u\|_{\ep}:=\langle u,u\rangle_{\ep}^{1/2}<\wq\big\}.
\]
And then we will construct $k$-peak solutions of Eq. \eqref{eq: Kirchhoff} of the forms
\[
u_{\ep}=\sum_{j=1}^{k}U_{\ep,y_{j}}+\var,
\]
where $U_{\ep,y_{j}}$ is the solution of limiting equation of \eqref{eq: Kirchhoff} which will be defined later.
So, Eq. \eqref{eq: Kirchhoff} can be rewritten as the following equation about $\var$:
\begin{equation}\label{rewrite equation}
\begin{cases}
L_{\ep}\var=l_{\ep}+R_{\ep}(\var), \ \ x\in\R^{N},\\
\var\in H^{1}(\R^{N}),
\end{cases}
\end{equation}
where the linear operator $L_{\ep}$,  the  terms $l_{\ep}$ and  $R_{\ep}(\var)$ are be defined in Section 2 Later.

In the traditional calculations, under the general $H^{1}(\R^{N})$ norm, we find
\begin{equation}\label{R}
  \|R_{\ep}(\var)\|_{\ep}=o(\|\var\|_{\ep}).
\end{equation}
Then, for $\var$ small, \eqref{rewrite equation} can be seen as a perturbation of the following problem
 \begin{equation}\label{linear equation}
  \begin{cases}
L_{\ep}\var=l_{\ep}, \ \ \ x\in\R^{N},\\
\var_{\ep}\in H^{1}(\R^{N}),
\end{cases}
 \end{equation}
Suppose that $L_{\ep}$ is a bounded invertible map in some suitable space, then \eqref{linear equation} has a solution
$\var_{\ep}=L_{\ep}^{-1}l_{\ep}$.  So we can use the contraction mapping theorem in the following small ball
\[
\left\{\var\in H^{1}(\R^{N}): \|\var\|_{\ep}\leq \ep^{\tau}\|l_{\ep}\|_{\ep}, 0<\tau<1 \right\}
\]
to solve \eqref{rewrite equation}.
While, for the logarithmic Schr\"{o}dinger equations \eqref{eq: Kirchhoff},
\begin{equation}\label{R1}
 |R_{\ep}(\var)|=O\Big(\var^{2}\big(\sum^k_{j=1}U_{\ep,y_{j}}\big)^{-1}\Big).
\end{equation}
In the general $H^{1}(\R^{N})$ space, $\|R_{\ep}(\var)\|_{\ep}$ isn't a higher order small term of $\|\var\|_{\ep}$, that is, \eqref{R} doesn't hold.
To overcome this difficulty, we define a new type of norm
\begin{equation}\label{new norm}
  \|\var\|_{*}=\sup_{x\in\R^{N}}\Big(\sum_{j=1}^{k}e^{-\frac{|x-y_{j}|^{2}}{2\ep^{2}}}\Big)^{-1}|\var(x)|,
\end{equation}
where $\varphi\in H_{\varepsilon}$, $y_{j}\in B_{\delta}(\xi_{j})$, and restrict $\var$ in the the following space
\begin{equation}\label{new space}
\wp_{\ep}:=\left\{\var\in H_{\ep}: \ \|\var\|_{*}\leq \frac{1}{|\ln\ep|^{1-\theta}}\right\},~\mbox{with some small}~\theta>0.
\end{equation}
Then we conduct the contraction mapping in a small ball $S$\big(see \eqref{SS}\big) endowed with the norm $\|\cdot\|_{*}$.

After this reduction progress, we only need to solve a finite dimensional problem about $y_{j}$. Different from the general minimum or maximum progress, inspired by  \cite{Peng-2018}, we use the Pohozaev identity of \eqref{eq: Kirchhoff} to ensure the existence of $y_{j}$. And this methods allow the peak points $y_{j}$ of $u_{\ep}$ can be the non-degenerate critical points of $V(x)$, not just minimum points or maximum points of $V(x)$.

We also consider  the local uniqueness of the $k$-peak solution of \eqref{eq: Kirchhoff}.

\begin{theorem}\label{thm: uniqueness}
  Assume that
$(V_{1})$ and $(V_{2})$ hold. If $
u^{(i)}_{\ep}$ with $i=1,2$ are the positive solution of  \eqref{eq: Kirchhoff}
 concentrated at $\xi_{1},\cdots,\xi_{k}$.
Then
$u_{\ep}^{(1)}\equiv u_{\ep}^{(2)}$ for $\ep$ sufficiently small.
\end{theorem}
\begin{remark}
 In Theorem \ref{thm: uniqueness} with  $k=1$, we find the uniqueness result about single-peak solution concentrated at  a non-degenerate critical point of $V(x)$.
On the other hand,  the ground state of \eqref{eq: Kirchhoff} must concentrate at a minimum point of $V(x)$. So if we impose an other condition on $V(x)$ as follows:
\begin{equation*}
V(\xi_1)=\inf_{\R^N}V(x)~\mbox{and}~V(x)>V(\xi_1)~\mbox{for any}~x\in \R^N\setminus \{\xi_1\}.
\end{equation*}
Then the ground state of \eqref{eq: Kirchhoff} is unique by Theorem \ref{thm: uniqueness}.
\end{remark}
We will prove Theorem \ref{thm: uniqueness} inspired by \cite{Cao-Li-Luo-2015}. Let $u_{\ep}^{(l)}$ with $l=1,2$ be two different positive solutions concentrated at $k$ points $\xi_{1},\cdots,\xi_{k}$.  Set
\[
\eta_{\ep}=\frac{u_{\ep}^{(1)}-u_{\ep}^{(2)}}{\|u_{\ep}^{(1)}-u_{\ep}^{(2)}\|_{L^{\wq}(\R^{N})}}.
\]
Then we prove $\|\eta_{\ep}\|_{L^{\wq}(\R^{N})}=o(1)$ to obtain a contradiction with  $\|\eta_{\ep}\|_{L^{\wq}(\R^{N})}=1$. We will use the blow-up analysis and local Pohozaev
type of identities to deal with the
estimate near the concentrated points. But we will use the
maximum principle for the calculations away from the concentrated points.

In this paper, we write $\int u$ to denote Lebesgue integrals over
$\R^{N}$, unless otherwise stated, $\|u\|_{p}=\big(\int u^{p}\big)^{\frac{1}{p}}$ and $\langle u,v\rangle=\int uv$.
 We will use $C$ to denote various positive
constants, and $O(t)$, $o(t)$ and $o(1)$ to mean $|O(t)|\le C|t|$, $o(t)/t\to0$
as $t\to0$ and $o(1)\rightarrow0$ as $\ep\rightarrow 0$, respectively.

The paper is organized as follows. In Section \ref{sec: form and location}
we give some notations and preliminary estimates.
 In Section \ref{Finite dimensional reduction}, we carry out the reduction argument.
In Sections \ref{Proof of the first Theorem} and \ref{Local uniqueness results}, we will complete the proofs of Theorems \ref{thm: main reuslt-existence} and  \ref{thm: uniqueness} correspondingly.

\section{Preliminaries \label{sec: form and location}}
From \cite{dAvenia-2014}, we know that $U(x):=e^{\frac{w+N-|x|^{2}}{2}}$ is the unique positive solution of the following problem
\begin{equation*}
-\De u+wu=u\log u^2, \ \  u>0, \ \  \text{in}~\R^{N}.
\end{equation*}
Furthermore, it is  non-degenerate in
$H^{1}(\R^{N})$ in the sense that
\[
\operatorname{Ker}\L={\rm span}\left\{ \frac{\pa U}{\pa x_{j}}:1\le j\le N\right\} ,
\]
where the linearized operator $\L:H^{1}(\R^{N})\to H^{1}(\R^{N})$ is defined as
\begin{equation*}
  \L\var\equiv-\De\var+(\omega-2-2\log U)\var,
~\mbox{for}~\var\in H^{1}(\R^{N}).
\end{equation*}
For any $y_{j}\in \R^{N}$ with $j=1,\cdots,k$, we
denote \[
U_{\ep,y_{j}}(x)=e^{\frac{V(y_{j})+N}{2}}e^{-\frac{|x-y_{j}|^{2}}{2\ep^{2}}},
\]
which is the solution of
\begin{equation}
-\ep^2\De U_{\ep,y_{j}}(x)+V(y_{j})U_{\ep,y_{j}}(x)=U_{\ep,y_{j}}(x)\log U_{\ep,y_{j}}^2(x)\ \ \text{in }\R^{N}.
\label{eq: our main part}
\end{equation}
The linearized operator of (\ref{eq: our main part}) at $U_{\ep,y_{j}}(x)$ is $\L_{\ep}:\equiv-\ep^{2}\De+V(y_{j})-2(\log U_{\ep,y_{j}}+1)$, whose kernel is
\[
K_{\ep}=span\big\{\frac{\partial U_{\ep,y_{j}}}{\partial x_{i}},\ i=1,\cdots,N,\ j=1,\cdots,k\big\}.
\]
We note $y=(y_1,\cdots,y_k)$ and
\[
E_{\ep,y}=\left\{v\in H_{\ep}:\left\langle v,\frac{\partial U_{\ep,y_{j}}}{\partial x_{i}}\right\rangle_{\ep}=0,\ i=1,\cdots,N,\ j=1,\cdots,k\right\}.
\]
Let $\xi_{j}(j=1,\cdots,k)$ be the critical points of $V(x)$, we want to construct a solution $u_{\varepsilon}$  to Eq. \eqref{eq: Kirchhoff} of the form
\begin{equation*}
u_{\varepsilon}(x)=\sum_{j=1}^{k}U_{\ep,y_{\ep,j}}(x)+\var_{\ep}(x),\label{3.11}
\end{equation*}
where $y_{\ep,j}\in\R^{N},\ \var_{\ep}\in E_{\ep,y}$ satisfies
\begin{equation*}
|y_{\ep,j}-\xi_{j}|=o(1),\ \  \|\varphi_{\ep}\|_{\varepsilon}=o(\varepsilon^{\frac{N}{2}}),\ j=1,\cdots,k\label{3.1-2}.
\end{equation*}
Then $\varphi_{\ep}$ satisfies the following equation:
\begin{equation}\label{new form equation}
\begin{cases}
L_{\ep}\var_{\ep}=l_{\ep}+R_{\ep}(\var_{\ep}),~x\in\R^{N},\\
\var_{\ep}\in H^{1}(\R^{N}),
\end{cases}
\end{equation}
where
\begin{equation}\label{linear part}
  L_{\ep}\var=-\ep^{2}\Delta\var+V(x)\var-2\Big(\log\big(\sum_{j=1}^{k}U_{\ep,y_{\ep,j}}\big)+1\Big)\var,
\end{equation}
\begin{equation}\label{zero order}
\begin{aligned}
  l_{\ep}=&\sum_{j=1}^{k}\big(V(y_{\ep,j})-V(x)\big)U_{\ep,y_{\ep,j}} +2\sum_{j=1}^{k} U_{\ep,y_{\ep,j}}\Big(
    \log\big(\sum_{t=1}^{k}U_{\ep,y_{\ep,t}}\big)- \log U_{\ep,y_{\ep,j}})\Big),
    \end{aligned}
\end{equation}
and
\begin{equation}\label{nonlinear term}
\begin{aligned}
R_{\ep}(\var)=&2\bigg[\Big(\sum_{j=1}^{k}U_{\ep,y_{\ep,j}}+\var\Big)
    \log\Big(\sum_{t=1}^{k}U_{\ep,y_{\ep,t}}+\var\Big)\\
&-\Big(\sum_{j=1}^{k}U_{\ep,y_{\ep,j}}\Big)
    \log\Big(\sum_{t=1}^{k}U_{\ep,y_{\ep,t}}\Big)-
    \Big(\log\big(\sum_{t=1}^{k}U_{\ep,y_{\ep,t}}\big)+1\Big)\var\bigg].
 \end{aligned}
\end{equation}

\noindent The procedure to construct a $k$-peak solution for \eqref{eq: Kirchhoff} consists
of two steps:

\smallskip

\noindent\textbf{Step {(1)}.}\emph{ Finite dimensional reduction: We solve (\ref{new form equation}) up to an
approximate kernel $K_{\ep} $ of $L_{\ep}$.
That is, for any given $y_{j}\in \R^{N}(j=1,\cdots,k)$, we prove
the existence of $\var_{\ep}\in E_{\ep,y}$, such that
\begin{equation}\label{new equation}
  L_{\ep}\var_{\ep}=l_{\ep}+R_{\ep}(\var_{\ep})+\sum_{j=1}^{k}\sum_{i=1}^{N}a_{\ep,i,j}\frac{\partial U_{\ep,y_{j}}}{\partial x_{i}},~
\mbox{for some constants}~a_{\ep,i,j}.
\end{equation}}

\noindent\textbf{Step {(2)}.} \emph{ Solve the finite dimensional problem. We need to choose $y_{j}$ suitably, such that all
the constants $a_{\ep,i,j}$ in (\ref{new equation}) are zero.}

 In order to use the contraction mapping theorem to carry out the reduction for $(\ref{new form equation})$, we need the following invertible result and estimate $\|l_{\ep}\|_{\ep}$ and $\|R_{\ep}(\var_{\ep})\|_{\ep}$.

\begin{proposition}\label{inversibility}
There exist $\ep_{1},\de_{1},\rho>0$, independent of $y_{j},\ j=1,\cdots,k$,
such that for any $\ep\in(0,\ep_{1}]$, $\de\in(0,\de_{1})$ and
$y_{j}\in B_{\de}(\xi_{j})$, $P_{\ep}L_{\ep}$ is bijective in $E_{\ep,y}$. Moreover, it holds
\begin{eqnarray*}
\|P_{\ep}L_{\ep}\var\|_{\ep}\ge\rho\|\var\|_{\ep}, &  & \var\in E_{\ep,y},
\end{eqnarray*}
with the projection $P_{\ep}$
from $H^{1}(\R^{N})$ to $E_{\ep,y}$ as follows:
\begin{equation}\label{propection}
  P_{\ep}u=u-\sum_{j=1}^{k}\sum_{i=1}^{N}\big\langle u,\frac{\partial U_{\ep,y_{j}}}{\partial x_{i}}\big\rangle\frac{\partial U_{\ep,y_{j}}}{\partial x_{i}}.
\end{equation}
\end{proposition}
\begin{proof}
We use a contradiction argument. Assume, on the contrary, that there
exist $\ep_{n}\to0$, $\de_{n}\to0$, $y_{n,j}\in B_{\de_{n}}(\xi_{j})$
and $\var_{n}\in E_{n}\equiv E_{\ep_{n},y_{n,j}}$ such that
\begin{eqnarray}
\langle P_{\ep_{n}}L_{\ep_{n}}\var_{n},\psi_{n}\rangle=o_{n}(1)\|\var_{n}\|_{\ep_{n}}\|\psi_{n}\|_{\ep_{n}}, &  & \forall\:\psi_{n}\in E_{n}.\label{eq: B.1}
\end{eqnarray}
Since the equality is homogeneous, we may assume, with no loss of
generality, that $\|\var_{n}\|_{\ep_{n}}=\ep_{n}^{N/2}$.
Using  (\ref{eq: B.1}), we get
\begin{equation}\label{A.2}
\begin{aligned}
\int & (\ep_{n}^{2}|\nabla\var_{n}|^{2}+V(x)\var_{n}^{2})-2\int\Big(\log \big(\sum_{j=1}^{k}U_{\ep_{n},y_{n,j}}\big)+1\Big)\var_{n}^{2}\\
 =&\langle L_{\ep_{n}}\var_{n},\var_{n}\rangle_{\ep}=\langle P_{\ep_{n}}L_{\ep_{n}}\var_{n},\var_{n}\rangle_{\ep}=
 o(1)\|\var_{n}\|^{2}_{\ep_{n}}=o(\ep_{n}^{N}).
\end{aligned}
\end{equation}
On the other hand, for $R>0$ large enough, we have
\[
2\log \big(\sum_{j=1}^{k}U_{\ep_{n},y_{n,j}}\big)
+3\leq\frac{1}{2}V(x),\ \ \text{in}\ \R^{N}\setminus\bigcup_{j=1}^{k}B_{\ep_{n}R}(y_{n,j}).
\]
So,
\begin{equation*}
\begin{aligned}
 \int&\ep_{n}^{2}|\nabla\var_{n}|^{2}+V(x)\var_{n}^{2}-2\int\Big(\log \big(\sum_{j=1}^{k}U_{\ep_{n},y_{n,j}}\big)+1\Big)\var_{n}^{2}\\
\geq&\ep^{N}_{n}-\int_{\bigcup_{j=1}^{k}B_{\ep_{n}R}(y_{n,j})}\Big(2\log \big(\sum_{j=1}^{k}U_{\ep_{n},y_{n,j}}\big)+3\Big)\var_{n}^{2}
-\int_{\R^{N}\setminus\bigcup_{j=1}^{k}B_{\ep_{n}R}(y_{n,j})}\frac{V(x)}{2}\var_{n}^{2}\\
\geq&\frac{1}{2}\ep^{N}_{n}-\int_{\bigcup_{j=1}^{k}B_{\ep_{n}R}(y_{n,j})}\Big(2\log \big(\sum_{j=1}^{k}U_{\ep_{n},y_{n,j}}\big)+3\Big)\var_{n}^{2}
\geq\frac{1}{2}\ep^{N}_{n}-C\int_{\bigcup_{j=1}^{k}B_{\ep_{n}R}(y_{n,j})}\var_{n}^{2}.
\end{aligned}
\end{equation*}
Combining with (\ref{A.2}), we get
\begin{equation}\label{A.4}
 \ep^{N}_{n}\leq o(\ep^{N}_{n})+C\int_{\bigcup_{j=1}^{k}B_{\ep_{n}R}(y_{n,j})}\var_{n}^{2}.
\end{equation}
To deduce contradiction from (\ref{A.4}),  we only need to prove
\begin{equation}\label{A.5}
  \int_{\bigcup_{j=1}^{k}B_{\ep_{n}R}(y_{n,j})}\var_{n}^{2}= o(\ep^{N}_{n}).
\end{equation}
For this purpose, we will discuss the local behaviors of $\var_{n}$ near each $y_{n,j}(j=1,\cdots,k)$. So we introduce
\[
\widetilde{\var}_{n,j}(x)=\var_{n}(\ep_{n}x+y_{n,j}).
 \]
Then, since $V(x)$ is bounded and $\inf_{\R^{N}}V>0$, we have
\[
\int\left(|\nabla\widetilde{\var}_{n,j}|^{2}+|\widetilde{\var}_{n,j}|^{2}\right)\leq C.
\]
 Hence, up to a subsequence, we may assume that
 \begin{eqnarray*}
 \widetilde{\var}_{n,j}\wto\var_{j},\quad\text{weakly in }H^{1}(\R^{N}),~~~
 \widetilde{\var}_{n,j}\to\var_{j},\quad\text{in }L_{\loc}^{q}(\R^{N}),\quad(1\le q<\frac{2N}{N-2}).
\end{eqnarray*}
for some $\var_{j}\in H^{1}(\R^{N})$, we will prove $\var_{j}\equiv0$.
Define
\[
\widetilde{E}_{\ep_{n}}=\left\{w:w\in H^{1}(\R^{N}),\left\langle w\left(\frac{x-y_{n,j}}{\ep_{n}}\right),
       \frac{\partial U_{\ep_{n},y_{n,j}}}{\partial x_{i}}\right\rangle_{\ep_{n}}=0,\ i=1,\cdots,N,j=1,\cdots,k\right\}.
\]
Now, for any $\widetilde{\phi}_{n,j}\in\widetilde{E}_{\ep_{n}}$, by (\ref{eq: B.1}), it holds
\begin{equation}
\begin{aligned}
\int& \nabla\widetilde{\var}_{n,j}\nabla\widetilde{\phi}_{n,j}+
\Big(V(\ep_{n}y+y_{n,j})-2 \log \big(\sum_{t=1}^{k}U_{\ep_{n},y_{n,t}}(\ep_{n}y+y_{n,j})\big)+1\Big)\widetilde{\var}_{n,j}\widetilde{\phi}_{n,j}\\
 =&\ep^{-N}_{n}\int\ep^{2}_{n}\nabla\var_{n,j}\nabla\phi_{n,j}+ \Big(V(x)  -2\log\big(\sum_{t=1}^{k}U_{\ep_{n},y_{n,t}}(\ep_{n}y+y_{n,j})\big)+1\Big)\var_{n,j}\phi_{n,j}\\
=&\ep^{-N}_{n}\langle\L_{\ep_{n}}\var_{n,j},\phi_{n,j}\rangle
=o(\ep^{-N}_{n})\|\var_{n,j}\|_{\ep_{n}}\|\phi_{n,j}\|_{\ep_{n}},
\end{aligned}
\label{A.6}
\end{equation}
 where $\phi_{n,j}(x)=\widetilde{\phi}_{n,j}(\frac{x-y_{n,j}}{\ep_{n}})\in E_{\ep_{n}}$.
 For any $\phi\in H^{1}(\R^{N})$, there exists $c_{\ep_{n},i,j}\in \R$ satisfying
 \[
\widetilde{\phi}_{n,j}=\phi-\sum_{j=1}^{k}\sum_{i=1}^{N}c_{\ep_{n},i,j}\frac{\partial U_{\ep_{n},y_{n,j}}(\ep_{n}x+y_{n,j})}{\partial x_{i}}\in\widetilde{E}_{\ep_{n}}.
 \]
If $\phi$ satisfies
\[
\int\nabla\phi\nabla\frac{\partial U_{\ep_{n},y_{n,j}}(\ep_{n}x+y_{n,j})}{\partial x_{i}}+V(y)\phi\frac{\partial U_{\ep_{n},y_{n,j}}(\ep_{n}x+y_{n,j})}{\partial x_{i}}=0,
\]
 for $i=1,\cdots,N,\ j=1,\cdots,k$, then $c_{\ep_{n},i,j}=0$.
Inserting $\widetilde{\phi}_{n,j}$ into (\ref{A.6})  and letting  $n\rightarrow\infty$, we find
 \begin{equation}\label{A.7}
 \int\nabla\var_{j}\nabla\phi+V(\xi_{j})\var_{j}\phi-2\int(\log U^{j}+1)\varphi_{j}\phi=0,
  \end{equation}
where $U^{j}=U_{\ep_{n},y_{n,j}}(\ep_{n}y+y_{n,j})=e^{\frac{V(\xi_{j})+N-|x|^{2}}{2}}$  satisfies
$
-\triangle U^{j}+V(\xi_{j})U^{j}=U^{j}\log (U^{j})^{2}.$

Furthermore, we know
\[
-\triangle \frac{\partial U^{j}}{\partial x_{i}}+V(\xi_{j})\frac{\partial U^{j}}{\partial x_{i}}
-2(\log U^{j}+1)\frac{\partial U^{j}}{\partial x_{i}}=0.
\]
And then \eqref{A.7} also holds for $\phi=\displaystyle\sum_{i=1}^{N}\frac{\partial U^{j}}{\partial x_{i}}$.
Thus, \eqref{A.7} holds for any $\phi\in H^{1}(\R^{N})$. So we have
\[
-\triangle\varphi_{j}+V(\xi_{j})\varphi_{j}-2(\log U^{j}+1)\varphi_{j}=0.
\]
Thus, the non-degeneracy  of $U^{j}$ gives $
\varphi_{j}=\displaystyle\sum_{i=1}^{N}c_{i}\frac{\partial U^{j}}{\partial x_{i}}$.

On the other hand, $\widetilde{\varphi}_{n,j}\in \widetilde{E}_{\ep_{n}}$ implies $\langle\varphi_{j}, \frac{\partial U^{j}}{\partial x_{i}}\rangle_{\ep}=0$ for any $i=1,\cdots,N$. As a result, $\varphi_{j}=0$  and thus (\ref{A.5})
follows. We complete the proof.
\end{proof}
\begin{lemma}\label{lem: estimate for the first order} Assume that
$\ensuremath{V}$ satisfies (V1) and (V2). Then, there exists a constant
$\ensuremath{C>0}$, independent of $\ensuremath{\ep,\de}$, such
that for any $\ensuremath{y_{j}\in B_{\de}(\xi_{j})}$ there holds
\begin{equation}\label{estimates of linear part}
 \|l_{\ep}\|_{\ep}=O\Big (\sum_{j=1}^{k}\big|\nabla V(y_{j})\big|\ep^{\frac{N}{2}+1}+\ep^{\frac{N}{2}+2}\Big).
\end{equation}
 \end{lemma}
\begin{proof}
From $(\ref{zero order})$, for any $\eta\in H_{\varepsilon}$, we have
\[
\begin{aligned}\langle l_{\ep},\eta\rangle_{\ep}
=&\int\sum_{j=1}^{k}\Big(V(y_{j})-V(x)\Big)U_{\ep,y_{j}}\eta+2\int\bigg(\sum_{j=1}^{k}U_{\ep,y_{j}}
   \Big(\log\big(\sum_{t=1}^{k}U_{\ep,y_{t}}\big)-\log U_{\ep,y_{j}}\Big)\bigg)\eta.
\end{aligned}
\]
As
\begin{equation}\label{example1}
\begin{aligned}
\int&\left(V(y_{j})-V(x)\right)U_{\ep,y_{j}}\eta\\=&O\Big(\int\Big(\nabla V(y_{j})(x-y_{j})+O(|x-y_{j}|^{2})\Big)U_{\ep,y_{j}}\eta\Big)\\
=&O\Big(\int\Big|\nabla V(y_{j})(x-y_{j})+O(|x-y_{j}|^{2})\Big|^{2}U^{2}_{\ep,y_{j}}\Big)^{\frac{1}{2}}\|\eta\|_{\ep}\Big)\\
=&O\left(\big|\nabla V(y_{j})\big|\ep^{\frac{N}{2}+1}+\ep^{\frac{N}{2}+2}\right)\|\eta\|_{\ep},
\end{aligned}
\end{equation}
and
\begin{equation}\label{example}
\begin{aligned}
 \int &U_{\ep,y_{j}}
   \Big(\log\big(\sum_{t=1}^{k}U_{\ep,y_{t}}\big)-\log U_{\ep,y_{j}}\Big) \eta\\
= &\int_{B_{\delta}(y_{j})}
 U_{\ep,y_{j}}\Big(\log\big(1+\frac{\sum_{t\neq j}U_{\ep,y_{t}}}{U_{\ep,y_{j}}}\big)\Big)\eta+\int_{\R^{N}\backslash B_{\delta}(y_{j} )}U_{\ep,y_{j}}\left(\log\frac{\sum_{t=1}^{k}U_{\ep,y_{t}}}{U_{\ep,y_{j}}}\right)\eta\\
=& O\Big(\int_{B_{\delta}(y_{j})}
         \Big(\sum_{t\neq j}U_{\ep,y_{t}}\Big)|\eta|
        + \int_{\R^{N}\backslash B_{\delta}(y_{j} )} U^{\frac{1}{2}}_{\ep,y_{j}}
         \Big(\sum_{t=1}^{k}U_{\ep,y_{t}}\Big)^{\frac{1}{2}}|\eta|\Big)=O\Big(e^{-\frac{c}{\ep^{2}}}\|\eta\|_{\ep}\Big),
\end{aligned}
\end{equation}
then we get (\ref{estimates of linear part}) from \eqref{example1} and \eqref{example}.
\end{proof}
\begin{lemma} \label{lem: error estimates}
It holds
\[
\|R_{\ep}(\var)\|_{\ep}=O\Big(\frac{1}{|\ln\ep|^{1-\theta}}\|\var\|_{\ep}\Big),~\mbox{for all}~
 \var\in \wp_{\ep},
\]
where $\wp_{\ep}$ was defined in \eqref{new space}.
 \end{lemma}

\begin{proof}
First, by \eqref{nonlinear term} and Taylor's expansion, we find  \eqref{R1}.
Then we can obtain
\[
\begin{aligned}
\langle R_{\ep}(\var),\eta\rangle_{\ep}=&O \Big(\int \Big(\sum_{j=1}^{k}e^{-\frac{|x-y_{j}|^{2}}{2\ep^{2}}}\Big)^{-1}|\var|\cdot|\var\eta|\Big)\\
=&O\Big(\|\varphi\|_{*}\int|\varphi \eta|\Big)
=O\Big(\frac{1}{|\ln\ep|^{1-\theta}}\|\var\|_{\ep}\|\eta\|_{\ep}\Big).
\end{aligned}
\] Thus we complete the proof.
\end{proof}

\section{Finite dimensional reduction}\label{Finite dimensional reduction}
In this section, we carry out the reduction argument.
For any fixed $y_{j}\in B_{\delta}(\xi_{j}),\ j=1,\cdots,k$, we consider the following problem:
\begin{equation}\label{project equation}
 P_{\ep}L_{\ep}\var=l_{\ep}+R_{\ep}(\var) ,\ \ \var\in E_{\ep,y}.
\end{equation}
\begin{lemma}
It holds
  \begin{equation}\label{l**}
   \|l_{\ep}\|_{*}=O\big(1\big),\ \ \ \ \|R_{\ep}(\var)\|_{*}=O\big(\|\var\|_{*}^{2}\big).
  \end{equation}
\end{lemma}
\begin{proof}
 Recall \eqref{zero order},
since $V(x)\in C^{1}$ satisfies $(V_{1})$, we have
 \[\begin{aligned}
&\sup_{x\in\R^{N}}   \Big|\sum_{j=1}^{k}\Big(V(y_{j})-V(x)\Big)U_{\ep,y_{\ep,j}}\Big|
         \Big(\sum_{j=1}^{k}e^{-\frac{|x-y_{j}|^{2}}{2\ep^{2}}}\Big)^{-1}\\
= &O\Big( \sup_{x\in\R^{N}}\Big|\sum_{j=1}^{k}\big(V(y_{j})-V(x)\big)U_{\ep,y_{\ep,j}}\Big|
         \Big(\sum_{j=1}^{k}\big|V(y_{j})-V(x)\big|e^{-\frac{|x-y_{j}|^{2}}{2\ep^{2}}}\Big)^{-1}\Big)
         =O\big(1\big).
   \end{aligned}\]
  Similar to \eqref{example},
\[
\Big|\Big(\sum_{j=1}^{k}U_{\ep,y_{j}}\Big)
          \log \Big(\sum_{t=1}^{k}U_{\ep,y_{t}}\Big)-\sum_{j=1}^{k}\Big(U_{\ep,y_{j}}\log U_{\ep,y_{j}}\Big)\Big|
          \Big(\sum_{j=1}^{k}e^{-\frac{|x-y_{j}|^{2}}{2\ep^{2}}}\Big)^{-1}=O\big(1\big).\]
Thus we get
$\|l_{\ep}\|_{*}=O\big(1\big)$. Also by \eqref{nonlinear term},
\[
|R_{\ep}(\var)|\Big(\sum_{j=1}^{k}e^{-\frac{|x-y_{j}|^{2}}{2\ep^{2}}}\Big)^{-1}
=O\Big( \var^{2}\Big(\sum_{j=1}^{k}U_{\ep,y_{j}}\Big)^{-1}
        \Big(\sum_{j=1}^{k}e^{-\frac{|x-y_{j}|^{2}}{2\ep^{2}}}\Big)^{-1}\Big)
=O\Big(\|\var\|_{*}^{2}\Big).\]
Then  we obtain $\|R_{\ep}(\var)\|_{*}=O\big(\|\var\|_{*}^{2}\big)$.
\end{proof}

\begin{proposition}\label{contract lemma}
  Assume $N\geq3$, $u$ solves
 \begin{equation}\label{project lemma}
  P_{\ep}L_{\ep}u= l_{\ep}+ R_{\ep}(\var) ,~u\in H^{1}(\R^{N}),
  \end{equation}
  with $\var\in E_{\ep,y}$ satisfying
$$ \|u\|_{\ep}=O(\ep^{\frac{N}{2}+1}),\   \|\var\|_{\ep}=O(\ep^{\frac{N}{2}+1})~\mbox{and}~\|\var\|_{*}\leq \frac{1}{|\ln\ep|^{1-\theta}},$$
 where $\theta>0$ is a small positive constant. Then it holds
  $$
  \|u\|_{*}\leq \frac{1}{|\ln\ep|^{1-\theta}}.
  $$
\end{proposition}

\begin{proof}
  From \eqref{propection} and \eqref{project lemma}, we have
  $$
  L_{\ep}u=l_{\ep}+R_{\ep}(\var)+\sum_{j=1}^{k}\sum_{i=1}^{N}a_{\ep,i,j}\frac{\partial U_{\ep,y_{j}}}{\partial x_{i}}, u\in H^{1}(\R^{N}).
  $$
Combining with the definition of $L_{\ep}$ in (\ref{linear part}), we get
\begin{equation*}
\begin{aligned}
 -\ep^{2}\Delta u=&(2-V(x))u+2u\log\big(\sum_{j=1}^{k}U_{\ep,y_{j}}\big)\\
 &+l_{\ep}+R_{\ep}(\var)+\sum_{j=1}^{k}\sum_{i=1}^{N}a_{\ep,i,j}\frac{\partial U_{\ep,y_{j}}}{\partial x_{i}}, \ u\in H^{1}(\R^{N}).
 \end{aligned}
\end{equation*}
Then we note
\begin{equation}\label{u expression}
  u(x)=u_{1}(x)+u_{2}(x)+u_{3}(x)+u_{4}(x)+u_{5}(x),
\end{equation}
with
$$
u_{1}(x)=\frac{1}{\ep^{2}}\frac{1}{N(N-2)\omega_N}\int\frac{1}{|z-x|^{N-2}}(2-V(z))u(z)dz,
$$
$$
u_{2}(x)=\frac{1}{\ep^{2}}\frac{1}{N(N-2)\omega_N}
\int\frac{2}{|z-x|^{N-2}}u(z)\log\bigg(\sum_{j=1}^{k}U_{\ep,y_{j}}(z)\bigg)dz,
$$
$$
u_{3}(x)=\frac{1}{\ep^{2}}\frac{1}{N(N-2)\omega_N}\int\frac{1}{|z-x|^{N-2}}l_{\ep}(z)dz,
$$
$$
u_{4}(x)=\frac{1}{\ep^{2}}\frac{1}{N(N-2)\omega_N}\int\frac{1}{|z-x|^{N-2}}R_{\ep}(\var(z))dz,
$$
$$
u_{5}(x)=\frac{1}{\ep^{2}}\frac{1}{N(N-2)\omega_N}\sum_{j=1}^{k}\sum_{i=1}^{N}a_{\ep,i,j}\int\frac{1}{|z-x|^{N-2}}\frac{\partial U_{\ep,y_{j}}}{\partial x_{i}}(z)dz.
$$
Now we estimate each term of (\ref{u expression}). We first give an elementary inequality
\begin{equation*}
  \sum_{j=1}^{k}a_{j}b_{j}\leq \Big(\sum_{j=1}^{k}a_{j}\Big)\cdot\Big(\sum_{j=1}^{k}b_{j}\Big),\ \ a_{j}, b_{j}>0,
\end{equation*}
which will be useful during the following process.\\
For $x\in \displaystyle\bigcap_{j=1}^{k}B^{c}_{R\ep}(y_{j})$, we have
\begin{equation*}
  \begin{aligned}
  |u_{1}(x)|
  &= \frac{1}{N(N-2)\omega_N\ep^{2}}\int\frac{1}{|z-x|^{N-2}}|2-V(z)||u(z)|dz\\
  &\leq\frac{C}{\ep^{2}}\int\frac{1}{|z-x|^{N-2}}|u(z)|dz\\
  &\leq\|u\|_{*}\frac{C}{\ep^{2}}\sum_{j=1}^{k}
  \int\frac{1}{|z-x|^{N-2}}e^{-\frac{|z-y_{j}|^{2}}{2\ep^{2}}}dz. \end{aligned}
\end{equation*}
Also $-\Delta \frac{1}{|z-x|^{N-2}}=\delta_z(x)$  in $\R^N$.
Let $w(x)=\frac{\ep^{4}}{|x-y_{j}|^{2}} e^{-\frac{|x-y_{j}|^{2}}{2\ep^{2}}}$, then
 $-\Delta w(x)\geq Ce^{-\frac{|x-y_{j}|^{2}}{2\ep^{2}}}$,
 \begin{equation}\label{estimate u1}
\int\frac{1}{|z-x|^{N-2}}e^{-\frac{|z-y_{j}|^{2}}{2\ep^{2}}}dz
\leq C\int \delta_z(x) w(z) dz=Cw(x)=C\frac{\ep^{4}}{|x-y_{j}|^{2}} e^{-\frac{|x-y_{j}|^{2}}{2\ep^{2}}}.
\end{equation}
This gives
\begin{equation*}
  \begin{aligned}
  |u_{1}(x)|\leq C\|u\|_{*}\sum_{j=1}^{k}\frac{\ep^{2}}{|x-y_{j}|^{2}} e^{-\frac{|x-y_{j}|^{2}}{2\ep^{2}}}, \ \ x\in \bigcap_{j=1}^{k}B^{c}_{R\ep}(y_{j}).
  \end{aligned}
\end{equation*}
 So we get
\begin{equation}\label{estimate u10}
\|u_{1}\|_{*}=O(\frac{1}{R^{2}})\|u\|_{*}, \ \ x\in \bigcap_{j=1}^{k}B^{c}_{R\ep}(y_{j}).
\end{equation}
For $x\in B_{R\ep}(y_{j}),\ j=1,\cdots,k$, we have
\begin{equation}\label{estimate u11}
  \begin{aligned}
  |u_{1}(x)|
=&O\Big({\ep^{-2}}\int\frac{1}{|z-x|^{N-2}}|2-V(z)||u(z)|dz\Big)\\
=&O\Big({\ep^{-2}}\Big(\int_{B^{c}_{2R\ep}(x)}+\int_{B_{2R\ep}(x)}\Big)\frac{1}{|z-x|^{N-2}}|u(z)|dz\Big)
=:A_{1}+B_{1}.
  \end{aligned}
\end{equation}
And then
\begin{equation}\label{estimate u12}
\begin{aligned}
A_{1}
=&O\Big({\ep^{-2}}\int_{B^{c}_{2R\ep}(x)}\frac{1}{|z-x|^{N-2}}|u(z)|dz\Big)=
O\Big({\ep^{-2}}\frac{1}{(2R\ep)^{N-2}}\int_{B^{c}_{2R\ep}(x)}|u(z)|dz\Big)\\
=&O\Big( \frac{\|u\|_{*}}{\ep^{N}(2R)^{N-2}}\int_{B^{c}_{R\ep}(y_{j})}
    \Big(\sum_{t=1}^{k}e^{-\frac{|z-y_{t}|^{2}}{2\ep^{2}}}\Big)dz\Big)
=O\Big(\frac{1}{R^{N-2}}\|u\|_{*}\Big),
\end{aligned}
\end{equation}
since $|z-y_{j}|\geq|z-x|-|x-y_{j}|\geq R\ep$ for $z\in B^{c}_{2R\ep}(x)$ and $x\in B_{R\ep}(y_{j})$. We find
\begin{equation}\label{estimate u13}
\begin{aligned}
B_{1}
&=O\Big(\ep^{-2}\Big(\int_{B_{2R\ep}(x)}\frac{1}{|z-x|^{p(N-2)}}\Big)^{\frac{1}{p}}
      \|u(z)\|_{2}^{\frac{2}{q}}
      \max_{z\in B_{2R\ep}(x)}|u(z)|^{1-\frac{2}{q}}\Big)\\
&=O\Big(\ep^{-2}\Big(\int_{0}^{2R\ep}\frac{r^{N-1}}{r^{N-\gamma}}\Big)^{\frac{1}{p}}
  \|u(z)\|_{\ep}^{\frac{2}{q}}\Big)
=O\Big( \ep^{\frac{\gamma}{p}-2}
  \cdot\ep^{(\frac{N}{2}+1-\tau)\frac{2}{q}}\Big)=O\Big(\ep^{\frac{4-2\gamma}{N-\gamma}(1-\tau)}\Big),
\end{aligned}
\end{equation}
where $\gamma>0$ small,\ $p=\frac{N-\gamma}{N-2}$ and $q=\frac{N-\gamma}{2-\gamma}$. So, by \eqref{estimate u11}-\eqref{estimate u13}, we know
\begin{equation}\label{estimate 1u101}
|u_{1}(x)|\Big(\sum_{j=1}^{k}e^{-\frac{|x-y_{j}|^{2}}{2\ep^{2}}}\Big)^{-1}=
O\big( \frac{1}{R^{N-2}}\|u\|_{*}+\ep^{\frac{2(2-\gamma)}{N-\gamma}(1-\tau)}\Big),\ \mbox{for}~\ x\in \bigcup_{j=1}^{k}B_{R\ep}(y_{j}).
\end{equation}
We conclude from \eqref{estimate u10} and \eqref{estimate 1u101} that
\begin{equation}\label{u1}
  \|u_{1}\|_{*}=O\Big(\frac{1}{R^{2}}\|u\|_{*}+
  \frac{1}{R^{N-2}}\|u\|_{*}+\ep^{\frac{2(2-\gamma)}{N-\gamma}(1-\tau)}\Big),\ x\in\R^{N} \mbox{and}~R\ \text{large enough}.
\end{equation}

Now we estimate $u_{2}(x)$. By a fact that for any $\alpha>0$,
\[
\Big|\log\Big(\sum_{s=1}^{k}U_{\ep,y_{s}}(z)\Big)\Big|\leq(1+\alpha)\Big|\log U_{\ep,y_{t}}(z)\Big|,\ \ \text{for any}\ \  t\in\{1,\cdots,k\},
\]
which gives
\[
\begin{aligned}
|u_{2}(x)|=&\frac{1}{\ep^{2}}\frac{1}{N(N-2)\omega_N}\int\frac{2}{|z-x|^{N-2}}u(z)
                      \Big|\log\Big(\sum_{j=1}^{k}U_{\ep,y_{j}}(z)\Big)\Big|dz\\
              \leq&\frac{1}{\ep^{2}}\frac{(1+\alpha)\|u\|_{*}}{N(N-2)\omega_N}\int\frac{2}{|z-x|^{N-2}}
                      \Big(\sum_{j=1}^{k}e^{-\frac{|z-y_{j}|^{2}}{2\ep^{2}}}\Big)
                      \Big|\log U_{\ep,y_{t}}(z)\Big|dz\\
             =&\frac{1}{\ep^{2}}\frac{(1+\alpha)\|u\|_{*}}{N(N-2)\omega_N}\int\frac{2}{|z-x|^{N-2}}
                      \Big(\sum_{j=1}^{k}e^{-\frac{|z-y_{j}|^{2}}{2\ep^{2}}}\Big)
                      \Big|\frac{V(y_{t})+N}{2}-\frac{|z-y_{t}|^{2}}{2\ep^{2}}\Big|dz  \\
              \leq&C\Big( {\ep^{-2}}\int\frac{2}{|z-x|^{N-2}}
                      \Big(\sum_{j=1}^{k}e^{-\frac{|z-y_{j}|^{2}}{2\ep^{2}}}\Big)dz\Big)\\
                    &+ \frac{1}{\ep^{2}}\frac{(1+\alpha)\|u\|_{*}}{N(N-2)\omega_N}
                    \sum_{j=1}^{k}\int\frac{1}{|z-x|^{N-2}}
                      \frac{|z-y_{j}|^{2}}{\ep^{2}}e^{-\frac{|z-y_{j}|^{2}}{2\ep^{2}}}dz=:u_{21}+u_{22}.
\end{aligned}
\]
For $x\in \displaystyle\bigcap_{j=1}^{k}B^{c}_{R\ep}(y_{j})$, by \eqref{estimate u1}, we have
\begin{equation}\label{u21}
\|u_{21}\|_{*}=O\Big(\frac{1}{R^{2}}\|u\|_{*}\Big),\ \ x\in \bigcap_{j=1}^{k}B^{c}_{R\ep}(y_{j}).
\end{equation}
while, if we denote $ \frac{(1+\alpha)}{N(N-2)\omega_N}=C_{N,\alpha}$,
\[
\begin{aligned}
u_{22}(x)=&C_{N,\alpha}\|u\|_{*}\sum_{j=1}^{k}\int\frac{2}{|z-x|^{N-2}}
                       \frac{|z-y_{j}|^{2}}{\ep^{4}}e^{-\frac{|z-y_{j}|^{2}}{2\ep^{2}}}dz\\
=&C_{N,\alpha}\|u\|_{*}\sum_{j=1}^{k}\Big(\int_{B^{c}_{R\ep}(y_{j})}+\int_{B_{R\ep}(y_{j})}\Big)\frac{2}{|z-x|^{N-2}}
                       \frac{|z-y_{j}|^{2}}{\ep^{4}}e^{-\frac{|z-y_{j}|^{2}}{2\ep^{2}}}dz
                       \\=:&u_{221}(x)+u_{222}(x).
\end{aligned}
\]
Take $R^{2}>2N$, for $z\in B^{c}_{R\ep}(y_{j})$, we get
\[\begin{aligned}
\frac{|z-y_{j}|^{2}}{\ep^{4}}e^{-\frac{|z-y_{j}|^{2}}{2\ep^{2}}}
=&\Delta e^{-\frac{|z-y_{j}|^{2}}{2\ep^{2}}}+\frac{N}{\ep^{2}}e^{-\frac{|z-y_{j}|^{2}}{2\ep^{2}}}\leq \Delta e^{-\frac{|z-y_{j}|^{2}}{2\ep^{2}}}
       + \frac{|z-y_{j}|^{2}}{2\ep^{4}}e^{-\frac{|z-y_{j}|^{2}}{2\ep^{2}}},
\end{aligned}\]
so we have
\[ \frac{|z-y_{j}|^{2}}{2\ep^{4}}e^{-\frac{|z-y_{j}|^{2}}{2\ep^{2}}}
\leq\Delta e^{-\frac{|z-y_{j}|^{2}}{2\ep^{2}}}.
\]
Then, we find
\[
u_{221}(x)
\leq C_{N,\alpha}\|u\|_{*}\sum_{j=1}^{k}\int_{B^{c}_{R\ep}(y_{j})}\frac{4}{|z-x|^{N-2}}\Delta e^{-\frac{|z-y_{j}|^{2}}{2\ep^{2}}}dz
\leq 4C_{N,\alpha}\|u\|_{*}\Big(\sum_{j=1}^{k}e^{-\frac{|x-y_{j}|^{2}}{2\ep^{2}}}\Big).
\]
On the other hand, by \eqref{estimate u1},
\[
\begin{aligned}
u_{222}(x)
              \leq&C_{N,\alpha}\|u\|_{*}\sum_{j=1}^{k}\int_{B_{R\ep}(y_{j})}\frac{1}{|z-x|^{N-2}}
                       \frac{R^{2}}{\ep^{2}}e^{-\frac{|z-y_{j}|^{2}}{2\ep^{2}}}dz\\
             =&\frac{C_{N,\alpha}R^{2}}{C(N,R)}\|u\|_{*}\sum_{j=1}^{k}\frac{1}{\ep^{2}}
                 \int_{B_{R\ep}(y_{j})}\frac{1}{|z-x|^{N-2}}\Delta\left(\frac{\ep^{4}}{|z-y_{j}|^{2}} e^{-\frac{|z-y_{j}|^{2}}{2\ep^{2}}}\right)dz\\
             =&\frac{C_{N,\alpha} R^{2}}{C(N,R)}\|u\|_{*}\sum_{j=1}^{k}
               \int_{B_{R\ep}(y_{j})}\delta_{z}(x)\frac{\ep^{2}}{|z-y_{j}|^{2}} e^{-\frac{|z-y_{j}|^{2}}{2\ep^{2}}}dz=0.
\end{aligned}
\]
As a result,
\begin{equation}\label{u222}
\|u_{22}\|_{*}\leq \frac{4(1+\alpha)}{N(N-2)\omega_N}\|u\|_{*}.
\end{equation}
Combing \eqref{u21} and \eqref{u222}, we get
\begin{equation}\label{u20}
\|u_{2}\|_{*}\leq\Big( \frac{C}{R^{2}}+\frac{4(1+\alpha)}{N(N-2)\omega_N}\Big)\|u\|_{*},\ \ R\ \ \text{large and} \ x\in \bigcap_{j=1}^{k}B^{c}_{R\ep}(y_{j}).
\end{equation}

For $x\in B_{R\ep}(y_{j}),\ j=1,\cdots,k$, similar to \eqref{estimate u11}-\eqref{estimate u13}, we can get
\begin{equation}\label{u200}
\big|u_{2}(x)\big|\Big(\sum_{j=1}^{k}e^{-\frac{|x-y_{j}|^{2}}{2\ep^{2}}}\Big)^{-1} \leq \frac{C}{R^{N-2}}\|u\|_{*}+ CR^{2}\ep^{\frac{2(2-\gamma)}{N-\gamma}(1-\tau)},\ \ x\in \bigcup_{j=1}^{k}B_{R\ep}(y_{j}).
\end{equation}
By \eqref{u20} and \eqref{u200}, we finally get
\begin{equation}\label{u2}
  \|u_{2}\|_{*}\leq\Big( \frac{C}{R^{2}}+ \frac{C}{R^{N-2}}+\frac{2(1+\alpha)}{N(N-2)\omega_N}\Big)\|u\|_{*}
  +CR^{2}\ep^{\frac{2(2-\gamma)}{N-\gamma}(1-\tau)},
\end{equation}
for suitably large $R$ and $\ep$ small.

\smallskip

Next we estimate $u_{3}$.
Recall (\ref{zero order}), we denote
\begin{equation*}
\begin{aligned}
  l_{\ep}&=\sum_{j=1}^{k}\big(V(y_{j})-V(x)\big)U_{\ep,y_{j}}\\
            &\ \ \ +2\Big(\big(\sum_{j=1}^{k}U_{\ep,y_{j}}\big)
    \log\big(\sum_{t=1}^{k}U_{\ep,y_{t}}\big)-\sum_{j=1}^{k}(U_{\ep,y_{j}}\log U_{\ep,y_{j}})\Big)\\
    &=:l_{\ep1}+l_{\ep2},
    \end{aligned}
\end{equation*}
and
\[
\begin{aligned}
u_{3}(x)&=\frac{1}{N(N-2)\omega_N\ep^{2}}\int\frac{1}{|z-x|^{N-2}}l_{\ep}(z)dz\\
              &=\frac{1}{N(N-2)\omega_N\ep^{2}}\Big[\sum_{j=1}^{k}\int\frac{1}{|z-x|^{N-2}}l_{\ep1}(z)dz
              + \sum_{j=1}^{k}\int\frac{1}{|z-x|^{N-2}}l_{\ep2}(z)dz\Big]\\
           &=:u_{31}(x)+u_{32}(x).
\end{aligned}
\]
For $x\in \displaystyle\bigcap_{j=1}^{k}B^{c}_{\ep\sqrt{|\ln\ep|}}(y_{j})$, we find
\[
\begin{aligned}
|u_{31}(x)|
=&O\Big({\ep^{-2}}\sum_{j=1}^{k}\int\frac{1}{|z-x|^{N-2}}|V(y_{j})-V(z)|U_{\ep,y_{j}}(z)dz\Big)\\
=&O\Big({\ep^{-2}}\sum_{j=1}^{k}\int\frac{1}{|z-x|^{N-2}}U_{\ep,y_{j}}(z)dz\Big).
\end{aligned}
\]
Then by \eqref{estimate u1},
we get
\[
|u_{31}(x)|\Big(\sum_{j=1}^{k}e^{-\frac{|x-y_{j}|^{2}}{2\ep^{2}}}\Big)^{-1}= O\Big(\sum_{j=1}^{k}\frac{\ep^{2}}{|x-y_{j}|^{2}}\Big)=O\Big(\frac{1}{|\ln\ep|}\Big), \ \  \ x\in \bigcap_{j=1}^{k}B^{c}_{\ep\sqrt{|\ln\ep|}}(y_{j}).
\]
Similarly, as
\[
\begin{aligned}
|u_{32}(x)|
=&O\Big({\ep^{-2}}\int\frac{1}{|z-x|^{N-2}}\sum_{j=1}^{k}U_{\ep,y_{j}}(z)
    \Big(\log\Big(\sum_{t=1}^{k}U_{\ep,y_{t}}(z)\Big)-\log U_{\ep,y_{j}}(z)\Big)dz\\
=&O\Big({\ep^{-2}}\int\frac{1}{|z-x|^{N-2}}\sum_{j=1}^{k}
U_{\ep,y_{j}}(z)\cdot\frac{\sum_{t=1}^{k}U_{\ep,y_{t}}(z)}{U_{\ep,y_{j}}(z)}dz\Big)\\
=&O\Big({\ep^{-2}}\sum_{t=1}^{k}\int\frac{k}{|z-x|^{N-2}}U_{\ep,y_{t}}(z)dz\Big),
\end{aligned}
\]
we also have
\[
|u_{32}(x)|\Big(\sum_{j=1}^{k}e^{-\frac{|x-y_{j}|^{2}}{2\ep^{2}}}\Big)^{-1}=O\Big(\frac{1}{|\ln\ep|}\Big), \ \  \ x\in \bigcap_{j=1}^{k}B^{c}_{\ep\sqrt{|\ln\ep|}}(y_{j}).
\]
These give us that
\begin{equation}\label{uu3}
\|u_{3}\|_{*}=O\big(\|u_{31}\|_{*}\big)+O\big(\|u_{32}\|_{*}\big)=O\Big(\frac{1}{|\ln\ep|}\Big), \ \  \ x\in \bigcap_{j=1}^{k}B^{c}_{\ep\sqrt{|\ln\ep|}}(y_{j}).
\end{equation}
Next, we consider the case $x\in B_{\ep\sqrt{|\ln\ep|}}(y_{j}),\ j=1,\cdots,k$.
\[
\begin{aligned}
|u_{3}(x)|
=&O\Big(\ep^{-2}\|l_{\ep}\|_{*} \sum_{j=1}^{k}
     \Big(\int_{B^{c}_{2\ep\sqrt{|\ln\ep|}}(x)}+
    \int_{B_{2\ep\sqrt{|\ln\ep|}}(x)}\Big)\frac{1}{|z-x|^{N-2}}
       e^{-\frac{|z-y_{j}|^{2}}{2\ep^{2}}}dz.
\end{aligned}
\]
Also by \eqref{l**}, we have
\[
\begin{aligned}
{\|l_{\ep}\|_{*}}&{\ep^{-2}}\sum_{j=1}^{k}
     \int_{B^{c}_{2\ep\sqrt{|\ln\ep|}}(x)}\frac{1}{|z-x|^{N-2}}
        e^{-\frac{|z-y_{j}|^{2}}{2\ep^{2}}}dz\\
=& O\Big(\frac{\|l_{\ep}\|_{*}}{\ep^{N}|\ln\ep|^{\frac{N-2}{2}}}
\sum_{j=1}^{k}
     \int_{B^{c}_{\ep\sqrt{|\ln\ep|}}(y_{j})}e^{-\frac{|z-y_{j}|^{2}}{2\ep^{2}}}dz\Big)=
O\Big(\frac{1}{|\ln\ep|^{\frac{N-2}{2}}}\Big),
\end{aligned}
\]
and
\[
\begin{aligned}
{\|l_{\ep}\|_{*}}&{\ep^{-2}}\sum_{j=1}^{k}
    \int_{B_{2\ep\sqrt{|\ln\ep|}}(x)}\frac{1}{|z-x|^{N-2}}
       e^{-\frac{|z-y_{j}|^{2}}{2\ep^{2}}}dz\\
=&O\Big( \|l_{\ep}\|_{*}{\ep^{-2}}\sum_{j=1}^{k}
    \Big(\int_{B_{2\ep\sqrt{|\ln\ep|}}(x)}\frac{1}{|z-x|^{p(N-2)}}\Big)^{\frac{1}{p}}
    \Big(\int_{B_{2\ep\sqrt{|\ln\ep|}}(y_{j})}e^{-\frac{q|z-y_{j}|^{2}}{2\ep^{2}}}\Big)^{\frac{1}{q}}\Big)\\
=&O\Big(\ep^{\frac{2(2-\gamma)}{N-\gamma}} |\ln\ep|^{\frac{\gamma(N-2)}{N-\gamma}}\Big),
\end{aligned}
\]
where $\gamma>0$ small,\ $p=\frac{N-\gamma}{N-2}$ and $q=\frac{N-\gamma}{2-\gamma}$.
So, we find
\begin{equation}\label{uu31}
\|u_{3}\|_{*}=O\Big( \frac{\ep^{\frac{1}{2}}}{|\ln\ep|^{\frac{N-2}{2}}}+\ep^{\frac{2(2-\gamma)}{N-\gamma}+\frac{1}{2}} |\ln\ep|^{\frac{\gamma(N-2)}{N-\gamma}}\Big),\ \ \ \ \ \ x\in B_{\ep\sqrt{|\ln\ep|}}(y_{j}),\ j=1,\cdots,k.
\end{equation}
From \eqref{uu3} and \eqref{uu31}, we get
\begin{equation}\label{u3}
  \|u_{3}\|_{*}=O\Big(\frac{1}{|\ln\ep|}\Big).
\end{equation}
By using \eqref{l**} and a similar estimate to $u_{3}$, we can get
\begin{equation}\label{u4}
  \|u_{4}\|_{*}=O\Big(\|\varphi\|_{*}^{2}\Big).
\end{equation}

Now similar to the estimate of $u_1$, we estimate $u_{5}$. First,
by $(\ref{estimates of linear part})$, we know
 \[
 |a_{\ep,i,j}|=O\Big(\ep^{-\frac{N}{2}+1}\|L_{\ep}u\|_{\ep}\Big)=O\Big(\ep^{-\frac{N}{2}+1}\|l_{\ep}\|_{\ep}\Big) = O\big(\ep^{2}\big).
 \]
On the other hand,  we have
\[
\begin{aligned}
|u_{5}(x)|=
&O\Big({\ep^{-4}}\sum_{j=1}^{k}\sum_{i=1}^{N}|a_{\ep,i,j}|\int  \frac{|z-y_{j}|}{|z-x|^{N-2}}
   e^{-\frac{|z-y_{j}|^{2}}{2\ep^{2}}}dz\Big)\\=
&O\Big({\ep^{-2}}\sum_{j=1}^{k} \int  \frac{|z-y_{j}|}{|z-x|^{N-2}}
   e^{-\frac{|z-y_{j}|^{2}}{2\ep^{2}}}dz\Big).
\end{aligned}
\]
For $x\in \displaystyle\bigcap_{j=1}^{k}B^{c}_{R\ep}(y_{j})$, let $w_1(x)=\frac{\ep^{4}}{|x-y_{j}|} e^{-\frac{|x-y_{j}|^{2}}{2\ep^{2}}}$, we find
 $-\Delta w_1(x)\geq C|x-y_{j}|e^{-\frac{|x-y_{j}|^{2}}{2\ep^{2}}}$ and then
 \begin{equation}\label{estimate u5}
\int\frac{|z-y_{j}|}{|z-x|^{N-2}}e^{-\frac{|z-y_{j}|^{2}}{2\ep^{2}}}dz
=O\Big(\frac{\ep^{4}}{|x-y_{j}|} e^{-\frac{|x-y_{j}|^{2}}{2\ep^{2}}}\Big).
\end{equation}
This gives
\begin{equation}\label{8-11}
  \begin{aligned}
  |u_{5}(x)|\leq C \sum_{j=1}^{k}\frac{\ep^{2}}{|x-y_{j}| } e^{-\frac{|x-y_{j}|^{2}}{2\ep^{2}}}\leq
 \frac{C\ep}{R}\sum_{j=1}^{k} e^{-\frac{|x-y_{j}|^{2}}{2\ep^{2}}}, \ \ x\in \bigcap_{j=1}^{k}B^{c}_{R\ep}(y_{j}).
  \end{aligned}
\end{equation}
For $x\in \displaystyle\bigcup_{j=1}^{k}B_{R\ep}(y_{j})$,  we have
\begin{equation}\label{estimate u51}
  |u_{5}(x)|
=O\Big({\ep^{-2}}\sum_{j=1}^{k}\Big(\int_{B^{c}_{2R\ep}(x)}+\int_{B_{2R\ep}(x)}\Big)   \frac{|z-y_{j}|}{|z-x|^{N-2}}
   e^{-\frac{|z-y_{j}|^{2}}{2\ep^{2}}}dz\Big)
=:A_{2}+B_{2},
\end{equation}
and then
\begin{equation}\label{estimate u52}
\begin{aligned}
A_{2}
=&O\Big({\ep^{-2}}\sum_{j=1}^{k}\int_{B^{c}_{2R\ep}(x)}\frac{|z-y_{j}|}{|z-x|^{N-2}}
   e^{-\frac{|z-y_{j}|^{2}}{2\ep^{2}}}dz\Big)\\=&
O\Big({\ep^{-2}}\sum_{j=1}^{k}\frac{1}{(2R\ep)^{N-2}}\int_{B^{c}_{R\ep}(y_{j})}|z-y_{j}|
   e^{-\frac{|z-y_{j}|^{2}}{2\ep^{2}}}dz\Big)
=O\Big(\frac{\ep}{R^{N-2}e^{R^2}}\Big).
\end{aligned}
\end{equation}
Also, we find
\begin{equation}\label{estimate u53}
\begin{aligned}
B_{2}
&=O\Big(\frac{R}{\ep}\int_{B_{2R\ep}(x)}\frac{1}{|z-x|^{N-2}}\Big)
=O\Big(R^3 \ep\Big).
\end{aligned}
\end{equation}
Then by \eqref{estimate u51}-\eqref{estimate u53}, we know
\begin{equation}\label{estimate u101}
|u_{5}(x)|\Big(\sum_{j=1}^{k}e^{-\frac{|x-y_{j}|^{2}}{2\ep^{2}}}\Big)^{-1}=O\Big(R^3e^{\frac{R^2}{2}} \ep\Big),\ \mbox{for}~\ x\in \bigcup_{j=1}^{k}B_{R\ep}(y_{j}).
\end{equation}
 So from \eqref{8-11} and \eqref{estimate u101}, we find
 \begin{equation}\label{u5}
 \|u_{5}\|_{*}=O(\ep).
 \end{equation}
Above all, from (\ref{u expression}), (\ref{u1}), (\ref{u2}), (\ref{u3}), (\ref{u4}), (\ref{u5}), we get
\[
\|u\|_{*}=O(\frac{1}{|\ln\ep|})\leq \frac{1}{|\ln\ep|^{1-\theta}}.
\]
\end{proof}
\begin{proposition} \label{prop: reduction map}
Assume $N\geq3$. Let $\delta>0$ be small such that $B_{\delta}(\xi_{i})\cap B_{\delta}(\xi_{j})=\emptyset$ for $i,j=1,\cdots,k$, $i\neq j$, there exists $\ep_{0}>0$ such that for any $\ep\in(0,\ep_{0}]$, $y_{j}\in B_{\delta}(\xi_{j})$, there is a unique map $\ensuremath{\var_{\ep,y}:B_{\delta}(\xi_{j})\to H_{\varepsilon}}$ with
$\ensuremath{y\mapsto\var_{\ep,y}\in E_{\ep,y}}$ satisfying (\ref{project equation}), where $y=(y_{1},\cdots,y_{k})$.
 Moreover,
\begin{equation}\label{estimate of small term}
\|\var_{\ep,y}\|_{\ep}\le C\|l_{\ep}\|_{\ep}
\le C\Big(\sum_{j=1}^{k}\big|\nabla V(y_{j})\big|\ep^{\frac{N}{2}+1}+\ep^{\frac{N}{2}+2}\Big),
\end{equation}
and
\begin{equation}
\|\var\|_{*}<\frac{1}{|\ln\ep|^{1-\theta}}.\label{*estimate}
\end{equation}
\end{proposition}
\begin{proof}
  By Proposition \ref{inversibility}, we can rewrite (\ref{project equation}) as
  \begin{equation*}\label{rewrite form}
 \var=B\var:=(P_{\ep}L_{\ep})^{-1}l_{\ep}+(P_{\ep}L_{\ep})^{-1}R_{\ep}(\var).
 \end{equation*}
  It follows from Proposition \ref{inversibility}  and (\ref{estimates of linear part}) that
\[
\|(P_{\ep}L_{\ep})^{-1}l_{\ep}\|_{\ep}\leq C\|l_{\ep}\|_{\ep}
\leq C\ep^{\frac{N}{2}+1}.
\]
Now we will apply the contraction mapping theorem in the set
\begin{equation}\label{SS}
S:=\left\{\var:\var\in E_{\ep,y}, \|\var\|_{\ep}\leq\ep^{\frac{N}{2}+1-\tau}, \ \|\var\|_{*}\leq\frac{1}{|\ln\ep|^{1-\theta}}\right\}
\end{equation}
endowed with the norm $\|\cdot\|_{*}$, where $\tau,\theta>0$ are some fixed small constants.

Then for any $\var_{1},\ \var_{2}\in S$, it holds
\[
\begin{aligned}
\|B\var_{1}-B\var_{2}\|_{*}
&\leq C\|R_{\ep}(\var_{1})-R_{\ep}(\var_{2})\|_{*} =C\|R'_{\ep}(\var_{1}+\theta(\var_{2}-\var_{1}))\cdot(\var_{1}-\var_{2})\|_{*}\\
&=C\|\log\left(1+\frac{\var_{1}+\theta(\var_{2}-\var_{1})}{\sum_{j=1}^{k}U_{\ep,y_{j}}}\right)\cdot(\var_{1}-\var_{2})\|_{*}\\
&\leq C\sum_{i=1}^{2}\|\var_{i}\|_{*}\cdot\|\var_{1}-\var_{2}\|_{*} \leq\frac{C}{|\ln\ep|^{1-\theta}}\|\var_{1}-\var_{2}\|_{*}
\leq \frac{1}{2}\|\var_{1}-\var_{2}\|_{*},
\end{aligned}
\]
where $\theta\in[0,1]$.
For any $\varphi\in  E_{\ep,y}$, by Lemma \ref{lem: estimate for the first order} and Lemma \ref{lem: error estimates}, we get
\begin{equation*}\label{concract}
\begin{aligned}
 &\|B\var\|_{\ep}
\leq C\|l_{\ep}\|_{\ep}+C\|R_{\ep}(\var)\|_{\ep} \leq C\ep^{\frac{N}{2}+1}+\frac{C}{|\ln\ep|^{1-\theta}}\|\var\|_{\ep}
 \leq \ep^{\frac{N}{2}+1-\tau}.
 \end{aligned}
\end{equation*}
On the other hand, applying Lemma \ref{contract lemma} to $u=B\var$, we have
\[
\|B\var\|_{*}\leq\frac{1}{|\ln\ep|^{1-\theta}}.
\]
So we get $B\var\in S$.
Then by the contraction mapping theorem, we conclude that for $\varepsilon,\ \delta$ sufficiently small, there exists
$\var_{\ep}\in E_{\ep, y}$ depending on $y$ and $\ep$, satisfying $
\var_{\ep}=B\var_{\ep}$.
Moreover, we know
\[
\|\var_{\ep}\|_{\ep}
=O\Big(\|l_{\ep}\|_{\ep}+ \|R_{\ep}(\var)\|_{\ep}\Big)
=O\Big(\|l_{\ep}\|_{\ep}+\frac{1}{|\ln\ep|^{1-\theta}}\|\var\|_{\ep}\Big),
\]
which gives
\[
\|\var_{\ep}\|_{\ep}=O\big(\|l_{\ep}\|_{\ep}\big)=O\Big(\sum_{j=1}^{k}\big|\nabla V(y_{j})\big|\ep^{\frac{N}{2}+1}+\ep^{\frac{N}{2}+2}\Big).
\]
\end{proof}

\section{Proof of Theorem \ref{thm: main reuslt-existence}}\label{Proof of the first Theorem}
 Theorem \ref{thm: main reuslt-existence} can be deduced from the following result.
\begin{theorem}\label{direct result}
Assume that $(V_{1})$ and $(V_{2})$ holds, $N\geq3$. Then, for $\ep>0$ sufficiently small, equation
\eqref{eq: Kirchhoff} has a solution of the form
\[
u_{\ep}=\sum_{j=1}^{k}U_{\ep,y_{\ep,j}}+\var_{\ep},
\]
for some $y_{\ep,j}\in B_{\delta}(\xi_{j})$, $\|\var_{\ep}\|_{\ep}=O(\ep^{\frac{N}{2}+1})$
 and $\|\var_{\ep}\|_{*}\leq\frac{1}{|\ln\ep|^{1-\theta}}$ with some small $\theta>0$.
\end{theorem}
First, Proposition \ref{prop: reduction map} implies the existence of $\var_{\ep}\in E_{\ep,y}$ , such that
\begin{equation}\label{new form}
  L_{\ep}\var_{\ep}-l_{\ep}-R_{\ep}(\var_{\ep})=\sum_{j=1}^{k}\sum_{i=1}^{N}a_{\ep,i,j}\frac{\partial U_{\ep,y_{j}}}{\partial x_{i}},
\end{equation}
for some constants $a_{\ep,i,j}$.
So we need to choose $y_{j}$ suitably such that
$a_{\ep,i,j}=0,\ i=1,\cdots,N,\ j=1,\cdots,k$.

The function in the right hand side of (\ref{new form}) belongs to
\[
E^{\bot}_{\ep,y}=span~\big\{\frac{\partial U_{\ep,y_{j}}}{\partial x_{i}},\ i=1,\cdots,N,\ j=1,\cdots,k\big\}.
\]
Therefore, we want to prove the left hand side of (\ref{new form}) belongs to $E_{\ep,y}$ , then the
function in the right hand side of (\ref{new form}) must be zero.

We first use the notation that
\[
u_{\varepsilon}=\sum_{j=1}^{k}U_{\ep,y_{j}}+\var_{\ep},
\]
Then, for any $\eta\in H_{\ep}$,
\[
\begin{aligned}
\langle L_{\ep}\var_{\ep}-l_{\ep}-R_{\ep}(\var_{\ep}),\eta\rangle=&\langle -\ep^{2}\Delta u_{\ep}+V(x)u_{\ep}-u_{\ep}\log u^{2}_{\ep},\eta\rangle\\=&\int(\ep^{2}\nabla u_{\ep}\nabla\eta+V(x)u_{\ep}\eta-u_{\ep}\eta\log u^{2}_{\ep}).
\end{aligned}
\]
\begin{lemma}
  Suppose that $y_{\ep,j}$ with $j=1,\cdots,k$ satisfies
  \begin{equation}\label{algebraic equation}
  \begin{aligned}
    &\int\Big(\ep^{2}\nabla u_{\ep}\nabla\frac{\partial U_{\ep,y_{\ep,j}}}{\partial x_{i}}+V(x)u_{\ep}\frac{\partial U_{\ep,y_{\ep,j}}}{\partial x_{i}}-u_{\ep}\frac{\partial U_{\ep,y_{\ep,j}}}{\partial x_{i}}\log u^{2}_{\ep}\Big)=0,\ \ i=1,\cdots,N.
    \end{aligned}
  \end{equation}
Then $a_{\ep,i,j}=0,\ i=1,\cdots,N,\ \ j=1,\cdots,k$.
\end{lemma}
 \begin{proof}
   If (\ref{algebraic equation}) holds, then
   \[
   \sum_{s=1}^{k}\sum_{m=1}^{N}a_{\ep,m,s}\left\langle\frac{\partial U_{\ep,y_{\ep,s}}}{\partial x_{m}},\frac{\partial U_{\ep,y_{\ep,j}}}{\partial x_{i}} \right\rangle=0,\ i=1,\cdots,N,\ \ j=1,\cdots,k.
   \]
Which implies that $a_{\ep,i,j}=0,\ i=1,\cdots,N,\ \ j=1,\cdots,k$.
 \end{proof}

\begin{proof}[Proof of Theorem \ref{direct result}]
  We only need to solve the algebraic equations (\ref{algebraic equation}).
 The main task is to find the
main term for the function in the left hand side of (\ref{algebraic equation}). The
procedure is that we first estimate the left hand side of (\ref{algebraic equation}) with
$\var_{\ep}=0$ . Then we show that the contribution of the error term $\var_{\ep}$ to the
function in the left hand side of (\ref{algebraic equation}) is negligible.

Denote $
G_{\ep,y}:=\displaystyle\sum_{j=1}^{k}U_{\ep,y_{j}}$.
  From (\ref{eq: our main part}) and the symmetry of $U_{\ep,y_{j}}$ we get
  \begin{equation*}\label{first estimate}
  \begin{aligned}
    \int \ep^{2}\nabla G_{\ep,y}\nabla\frac{\partial U_{\ep,y_{j}}}{\partial x_{i}}
    =&-\int\sum_{s=1}^{k} V(y_{s})U_{\ep,y_{s}}\frac{\partial U_{\ep,y_{j}}}{\partial x_{i}}+\int\sum_{s=1}^{k} U_{\ep,y_{s}}\frac{\partial U_{\ep,y_{j}}}{\partial x_{i}}\log U^{2}_{\ep,y_{s}}\\
    =&-\int\sum_{s\neq j} V(y_{s})U_{\ep,y_{s}}\frac{\partial U_{\ep,y_{j}}}{\partial x_{i}}
       +\int\sum_{s\neq j} U_{\ep,y_{s}}\frac{\partial U_{\ep,y_{j}}}{\partial x_{i}}\log U^{2}_{\ep,y_{s}}\\
    =&O(e^{-\frac{c}{\ep^{2}}}),
    \end{aligned}
  \end{equation*}
  and
\[
\begin{aligned}
\int V(x)G_{\ep,y}\frac{\partial U_{\ep,y_{j}}}{\partial x_{i}}
=&\int (V(x)-V(y_{j}))U_{\ep,y_{j}}\frac{\partial U_{\ep,y_{j}}}{\partial x_{i}}+O(e^{-\frac{c}{\ep^{2}}})\\
=&\frac{1}{2} \int (\frac{\partial V(y_{j})}{\partial x_{i}}(x_{i}-y_{j,i})+O(|x-y_{j}|^{2}))\frac{\partial U^{2}_{\ep,y_{j}}}{\partial x_{i}}+O(e^{-\frac{c}{\ep^{2}}})\\
=&\frac{1}{2}\ep^{N}  \frac{\partial V(y_{j})}{\partial x_{i}}\int x_{i}\frac{\partial U^{2}_{\ep,y_{j}}(\ep x+y_{j})}{\partial x_{i}}+O(\ep^{N+1})\\
=&-\frac{1}{2}\ep^{N}\frac{\partial V(y_{j})}{\partial x_{i}}\int U^{2}_{\ep,y_{j}}(\ep x+y_{j})
       +O(\ep^{N+1}),
\end{aligned}
\]
for some $c>0$.
Moreover, similar to (\ref{example}), we have
\[
\int G_{\ep,y}\frac{\partial U_{\ep,y_{j}}}{\partial x_{i}}\log G^{2}_{\ep,y}
=2\int \sum_{s=1}^{k}U_{\ep,y_{s}}\frac{\partial U_{\ep,y_{j}}}{\partial x_{i}}
  \left(\log G_{\ep,y}-\log U_{\ep,y_{s}}\right)+O(e^{-\frac{c}{\ep^{2}}})=O(e^{-\frac{c}{\ep^{2}}}).
\]
From above, we obtain
  \begin{equation*}\label{third estimate}
  \begin{aligned}
\int &\Big(\ep^{2}\nabla G_{\ep,y}\nabla\frac{\partial U_{\ep,y_{j}}}{\partial x_{i}}
  +V(x)G_{\ep,y}\frac{\partial U_{\ep,y_{j}}}{\partial x_{i}}
  -G_{\ep,y}\frac{\partial U_{\ep,y_{j}}}{\partial x_{i}}\log G^{2}_{\ep,y}\Big)\\
=&-\frac{1}{2}\ep^{N}\frac{\partial V(y_{j})}{\partial x_{i}}\int U^{2}_{\ep,y_{j}}(\ep x+y_{j})
       +O(\ep^{N+1}).
    \end{aligned}
  \end{equation*}

Now we show that the contribution of the error term $\var_{\ep}$ to the
function in the left hand side of (\ref{algebraic equation}) is negligible.

As  $\var_{\ep}\in E_{\ep,y}$, for $i=1,\cdots,N,\ j=1,\cdots,k$, we have
\[
\begin{aligned}
\int&\Big(\ep^{2}\nabla (G_{\ep,y}+\var_{\ep})\nabla\frac{\partial U_{\ep,y_{j}}}{\partial x_{i}}
  +V(x)(G_{\ep,y}+\var_{\ep})\frac{\partial U_{\ep,y_{j}}}{\partial x_{i}}\Big)\\
 =&\int\Big(\ep^{2}\nabla G_{\ep,y}\nabla\frac{\partial U_{\ep,y_{j}}}{\partial x_{i}}
  +V(x)G_{\ep,y}\frac{\partial U_{\ep,y_{j}}}{\partial x_{i}}\Big).
  \end{aligned}
\]
On the other hand,
\[
\begin{aligned}
2&\int\Big (G_{\ep,y}+\var_{\ep}\Big)\frac{\partial U_{\ep,y_{j}}}{\partial x_{i}}
 \log\Big(G_{\ep,y}+\var_{\ep}\Big)\\
 =& 2\int G_{\ep,y}\frac{\partial U_{\ep,y_{j}}}{\partial x_{i}}\log G_{\ep,y}
  +2\int(\log G_{\ep,y}+1)\frac{\partial U_{\ep,y_{j}}}{\partial x_{i}}\var_{\ep}
  +O\Big(\int \frac{\var_{\ep}}{G_{\ep,y}+\theta\var_{\ep}}
          \frac{\partial U_{\ep,y_{j}}}{\partial x_{i}}\var_{\ep}\Big).
  \end{aligned}
\]
First from (\ref{eq: our main part}) and $\var_{\ep}\in E_{\ep,y}$, we have
\[
\begin{aligned}
 2\int&(\log U_{\ep,y_{j}}+1)\frac{\partial U_{\ep,y_{j}}}{\partial x_{i}}\var_{\ep}
\\ =&\int\Big(\ep^{2}\nabla\frac{\partial U_{\ep,y_{j}}}{\partial x_{i}}\nabla\var_{\ep}+V(y_{j})\frac{\partial U_{\ep,y_{j}}}{\partial x_{i}}\var_{\ep}\Big)=\int(V(y_{j})-V(x))\frac{\partial U_{\ep,y_{j}}}{\partial x_{i}}\var_{\ep}\\
 =&O(|\nabla V(y_{j})|\ep^{\frac{N}{2}}+\ep^{\frac{N}{2}+1})\|\var_{\ep}\|_{\ep}=O(|\nabla V(y_{j})|\ep^{N+1}+\ep^{N+2}).
  \end{aligned}
\]
So, similar to (\ref{example}), we get
\[
\begin{aligned}
2\int&(\log G_{\ep,y}+1)\frac{\partial U_{\ep,y_{j}}}{\partial x_{i}}\var_{\ep}\\
 =& 2\int[(\log G_{\ep,y}+1)-(\log U_{\ep,y_{j}}+1)]\frac{\partial U_{\ep,y_{j}}}{\partial x_{i}}\var_{\ep}
   +2\int(\log U_{\ep,y_{j}}+1)\frac{\partial U_{\ep,y_{j}}}{\partial x_{i}}\var_{\ep}\\
 =&\int(\log G_{\ep,y}-\log U_{\ep,y_{j}})\frac{\partial U_{\ep,y_{j}}}{\partial x_{i}}\var_{\ep}
    +2\int(\log U_{\ep,y_{j}}+1)\frac{\partial U_{\ep,y_{j}}}{\partial x_{i}}\var_{\ep}\\
 =&O(e^{-\frac{c}{\ep^{2}}})+O(|\nabla V(y_{j})|\ep^{N+1}+\ep^{N+2})=O(|\nabla V(y_{j})|\ep^{N+1}+\ep^{N+2}).
  \end{aligned}
\]For the other term, we have
\[
\begin{aligned}
\int \frac{\var_{\ep}}{G_{\ep,y}+\theta\var_{\ep}}
          \frac{\partial U_{\ep,y_{j}}}{\partial x_{i}}\var_{\ep}
&=O\Big(\int G^{-1}_{\ep,y}\var_{\ep}\frac{\partial U_{\ep,y_{j}}}{\partial x_{i}}\var_{\ep}\Big)=O\Big(\|\var_{\ep}\|_{*}\int \frac{\partial U_{\ep,y_{j}}}{\partial x_{i}}\var_{\ep}\Big)\\&
=O\Big(\|\var_{\ep}\|_{*}|\nabla U_{\ep,y_{j}}|_{2}\|\var_{\ep}\|_{\ep}\Big) =O\Big(\|\var_{\ep}\|_{*}\ep^{\frac{N}{2}-1}\ep^{\frac{N}{2}+1} \Big)\\
& =O\Big(\frac{\ep^{N}}{|\ln\ep|^{1-\theta}}\Big).
  \end{aligned}
\]
So we get
\[
\begin{aligned}
 \int&\Big(\ep^{2}\nabla u_{\ep}\nabla\frac{\partial U_{\ep,y_{j}}}{\partial x_{i}}
  +V(x)u_{\ep}\frac{\partial U_{\ep,y_{j}}}{\partial x_{i}}-
 u_{\ep}\frac{\partial U_{\ep,y_{j}}}{\partial x_{i}}
 \log u_{\ep}\Big)\\
 =&\int\Big(\ep^{2}\nabla G_{\ep,y}\nabla\frac{\partial U_{\ep,y_{j}}}{\partial x_{i}}
  +V(x)G_{\ep,y}\frac{\partial U_{\ep,y_{j}}}{\partial x_{i}}-2\int G_{\ep,y}\frac{\partial U_{\ep,y_{j}}}{\partial x_{i}}\log G_{\ep,y}\Big)
  +O\Big(\frac{\ep^{N}}{|\ln\ep|^{1-\theta}}\Big)\\
 =&-\frac{1}{2}\ep^{N}\frac{\partial V(y_{j})}{\partial x_{i}}\int U^{2}_{\ep,y_{j}}(\ep x+y_{j})
 +O\Big(\frac{\ep^{N}}{|\ln\ep|^{1-\theta}}\Big),\ \ \  i=1,\cdots,N.
\end{aligned}
\]
As a result,  (\ref{algebraic equation}) is equivalent to
\begin{equation}\label{critical term}
  \frac{\partial V(y_{j})}{\partial x_{i}}=O\Big(\frac{1}{|\ln\ep|^{1-\theta}}\Big),\ \ \  i=1,\cdots,N.
\end{equation}
By \eqref{critical term} and the assumption $(V_{2})$, we have
\[
\frac{\partial^{2}V(\xi_{j})}{\partial\xi_{j,i}\partial\xi_{j,l}}(y_{j}-\xi_{j})+o(|y_{j}-\xi_{j}|)
=O\Big(\frac{1}{|\ln\ep|^{1-\theta}}\Big),\ \ \  i,l=1,\cdots,N.
\]
Then  (\ref{critical term}) has a solution
$y_{\ep,j}\in B_{\delta}(\xi_{j})$. We complete the proof.
\end{proof}

\section{Local uniqueness results}\label{Local uniqueness results}
In this section, we prove the local uniqueness result  Theorem \ref{thm: uniqueness}.
First, we give an important estimate on $|y_{\ep,j}-\xi_{j}|$, which can be improved  by using a class of Pohozaev type identities.
And the crucial Pohozaev type identities we will use are as follows:
\begin{proposition} \label{prop: Pohozaev identity}Let $u$ be a
 positive solution of Eq. \eqref{eq: Kirchhoff}. Let $\Om$ be a bounded
 smooth domain in $\R^{N}$. Then, for each $i=1,\cdots,N$, there hold
 \begin{equation}
 \begin{aligned}\int_{\Om}\frac{\pa V(x)}{\pa x_{i}}u^{2} & =\int_{\pa\Om}\Big[ \ep^{2}\Big(|\na u|^{2}\nu_{i}-2\frac{\pa u}{\pa\nu}\frac{\pa u}{\pa x_{i}}\Big)  + (V(x)+1)u^{2}\nu_{i}\Big]-\int_{\pa\Om}\nu_{i}u^{2}\log u^{2},
 \end{aligned}
\label{eq: Pohozaev}
\end{equation}
 where $\nu=(\nu_{1},\cdots,\nu_{N})$ is the unit outward normal of $\pa\Om$.
\end{proposition}

 Proposition \ref{prop: Pohozaev identity} can be directly proved by multiplying both sides of Eq. \eqref{eq: Kirchhoff} by $\frac{\partial u}{\partial x_{i}}$ and then integrating by parts.  Next, similar to Proposition 2.2 in \cite{Luo-Peng-Wang}, we find
\begin{lemma}\label{estimate about u}
  If $\var_{\ep}$ in Theorem \ref{direct result} satisfies $\|\var_{\ep}\|_{\ep}=o(\ep^{\frac{N}{2}})$, then there exists a small constant $\tau>0$, such that
  \begin{equation*}\label{var differential}
   |\var_{\ep}(x)|+|\nabla \var_{\ep}(x)|=O\big(e^{-\frac{\tau}{\ep}}\big), ~ \mbox{for}~x\in\R^{N}\setminus \displaystyle\bigcup^k_{j=1} B_{\tau}(y_{\ep,j}).
  \end{equation*}

\end{lemma}

\begin{proposition} \label{prop: rate of concentration} Let
$u_{\ep}=\displaystyle\sum_{j=1}^{k}U_{\ep,y_{\ep,j}}+\var_{\ep}$
be a solution of \eqref{eq: Kirchhoff}.
Then
\begin{eqnarray}\label{asymptotic behavior 2.1}
|y_{\ep,j}-\xi_{j}|=o(\ep).
\end{eqnarray}
\end{proposition}

\begin{proof}

 Let $u=u_{\ep},\ \Omega=B_{\delta}(y_{\ep,j})$ in \eqref{eq: Pohozaev},  we obtain
  \begin{equation}
 \begin{aligned}\int_{B_{\delta}(y_{\ep,j})}\frac{\pa V(x)}{\pa x_{i}}u_{\ep}^{2} & =\ep^{2}\int_{\pa B_{\delta}(y_{\ep,j})}\Big(|\na u_{\ep}|^{2}\nu_{i}-2\frac{\pa u_{\ep}}{\pa\nu}\frac{\pa u_{\ep}}{\pa x_{i}}\Big) \\
 & \quad+\int_{\pa B_{\delta}(y_{\ep,j})}(V(x)+1)u_{\ep}^{2}\nu_{i}-\int_{\pa B_{\delta}(y_{\ep,j})}\nu_{i}u_{\ep}^{2}\log u_{\ep}^{2},
 \end{aligned}
\label{eq: Pohozaev1}
\end{equation}
From Lemma \ref{estimate about u}, we have
\begin{equation}\label{boundary estimate}
 |u_{\ep}|+|\nabla u_{\ep}|\leq Ce^{-\frac{\gamma}{\ep}},\ \ \forall x\in\partial B_{\delta}(y_{\ep,j}),\ j=1,\cdots,k,
\end{equation}
here and in what follows $\gamma>0$ denote a constant which may change from line to line.
 By \eqref{boundary estimate}, for $x\in\partial B_{\delta}(y_{\ep,j})$, we find
 \[
\big| u_{\ep}^{2}\log u_{\ep}^{2}\big|\leq Ce^{-\frac{2\gamma}{\ep}}\big(\frac{2\gamma}{\ep}-\log C\big)
 =O(e^{-\frac{\gamma}{\ep}}).
 \]
So, \eqref{eq: Pohozaev1} equivalent to
 \begin{equation}\label{2.13}
   \int_{B_{\delta}(y_{\ep,j})}\frac{\pa V(x)}{\pa x_{i}}u_{\ep}^{2}=O(e^{-\frac{\gamma}{\ep}}).
 \end{equation}

On the other hand,
\begin{equation}
 \begin{aligned}
\int_{B_{\delta}(y_{\ep,j})}&\left(\frac{\pa V(x)}{\pa x_{i}}-\frac{\pa V(y_{\ep,j})}{\pa x_{i}}\right)u_{\ep}^{2}\\
  =&\int_{B_{\delta}(y_{\ep,j})}\langle\nabla^{2}V(y_{\ep,j}),x-y_{\ep,j}\rangle u_{\ep}^{2}
     +O\left(\int_{B_{\delta}(y_{\ep,j})}|x-y_{\ep,j}|^{2}u_{\ep}^{2}\right)\\
  =&\int_{B_{\delta}(y_{\ep,j})}\langle\nabla^{2}V(y_{\ep,j}),x-y_{\ep,j}\rangle
  (U_{\ep,y_{\ep,j}}^{2}+2U_{\ep,y_{\ep,j}}\var_{\ep}+\var_{\ep}^{2})
  +O(e^{-\frac{\gamma}{\ep}}+\ep^{N+2}).
 \end{aligned}
\label{cp left}
\end{equation}
Here we use Lemma \ref{estimate about u}.
Now, by the symmetry of $U_{\ep,y_{\ep,j}}$, we have
\[
\int_{B_{\delta}(y_{\ep,j})}\langle\nabla^{2}V(y_{\ep,j}),x-y_{\ep,j}\rangle U^{2}_{\ep,y_{\ep,j}}=0.
\]
By H\"{o}lder inequality and \eqref{estimate of small term}, we can get
\[
\int_{B_{\delta}(y_{\ep,j})}\langle\nabla^{2}V(y_{\ep,j}),x-y_{\ep,j}\rangle2U_{\ep,y_{\ep,j}}\var_{\ep}
+
\int_{B_{\delta}(y_{\ep,j})}\langle\nabla^{2}V(y_{\ep,j}),x-y_{\ep,j}\rangle\var_{\ep}^{2}=o(\ep^{N+1}).
\]
Inserting above into \eqref{cp left} and combine with \eqref{2.13}, we obtain
\[
\int_{B_{\delta}(y_{\ep,j})}\frac{\pa V(y_{\ep,j})}{\pa x_{i}}u_{\ep}^{2}=o(\ep^{N+1}).
\]
Then, for $ l=1,\cdots,N$,
\[
\int_{B_{\delta}(y_{\ep,j})}\langle\frac{\nabla^{2} V(\xi_{j})}{\pa x_{i}\pa x_{l}},y_{\ep,j,l}-\xi_{j,l}\rangle u_{\ep}^{2}=o(\ep^{N+1}).
\]
So, combining the condition $(V_{2})$ and
$
\displaystyle\int_{B_{\delta}(y_{\ep,j})} u_{\ep}^{2}=O(\ep^{N})$,
we get \eqref{asymptotic behavior 2.1}.
 \end{proof}

\begin{lemma}  Assume
$u_{\ep}=\displaystyle\sum_{j=1}^{k}U_{\ep,y_{\ep,j}}+\var_{\ep}$
be a solution of \eqref{eq: Kirchhoff}.
Then
\begin{equation*}\label{small quantity}
  \|\var_{\ep}\|_{\ep}=O(\ep^{\frac{N}{2}+2}).
\end{equation*}
\end{lemma}
\begin{proof}
First, we know the following  property
\begin{eqnarray}\label{property}
 \rho\|\var_{\ep}\|_{\ep}^2\leq\langle L_{\ep}\var_{\ep}, \var_{\ep}\rangle,  &  & \var_{\ep}\in E_{\ep,y}.
\end{eqnarray}
As the proof of (\ref{property}) is standard (see e.g. \cite{Cao-Noussair-Yan-1999}), we omit the details.
We mainly estimate $\langle L_{\ep}\var_{\ep}, \var_{\ep}\rangle$.
From \eqref{new form equation}, we have
\begin{equation*}\label{compute property}
  \langle L_{\ep}\var_{\ep}, \var_{\ep}\rangle=\int l_{\ep}\var_{\ep}+\int R_{\ep}(\var_{\ep})\var_{\ep},
\end{equation*}
where $L_{\ep}\var_{\ep},\ l_{\ep}$ and $R_{\ep}(\var_{\ep})$ are defined in \eqref{linear part}--\eqref{nonlinear term}. By \eqref{example}, we get
\[
\int 2\Big(\big(\sum_{j=1}^{k}U_{\ep,y_{\ep,j}}\big)
    \log\big(\sum_{t=1}^{k}U_{\ep,y_{\ep,t}}\big)-\sum_{j=1}^{k}(U_{\ep,y_{\ep,j}}\log U_{\ep,y_{\ep,j}})\Big)\var_{\ep}=O(e^{-\frac{c}{\ep^{2}}}\|\var_{\ep}\|_{\ep}).
\]
Under the condition $(V_{2})$, we obtaian
\[  \int\sum_{j=1}^{k}\big(V(y_{\ep,j})-V(x)\big)U_{\ep,y_{\ep,j}}\var_{\ep}
=\ep^{\frac{N}{2}}O(\ep^{2}+\ep|y_{\ep,j}-\xi_{j}|)\|\var_{\ep}\|_{\ep}.
 \]
 So, we find
\begin{equation}\label{5.11}
 \int l_{\ep}\var_{\ep}=\ep^{\frac{N}{2}}O(\ep^{2}+\ep|y_{\ep,j}-\xi_{j}|)\|\var_{\ep}\|_{\ep}.
\end{equation}

By \eqref{nonlinear term} and \eqref{*estimate}, we have
\begin{equation}\label{5.12}
\begin{aligned}
\int R_{\ep}(\var_{\ep})\var_{\ep}=
&\int2\bigg[\Big(\sum_{j=1}^{k}U_{\ep,y_{\ep,j}}+\var_{\ep}\Big)
    \log\Big(\sum_{t=1}^{k}U_{\ep,y_{\ep,t}}+\var_{\ep}\Big)\\
&-\Big(\sum_{j=1}^{k}U_{\ep,y_{\ep,j}}\Big)
    \log\Big(\sum_{t=1}^{k}U_{\ep,y_{\ep,t}}\Big)-
    \Big(\log\big(\sum_{t=1}^{k}U_{\ep,y_{\ep,t}}\big)+1\Big)\var_{\ep}\bigg]\var_{\ep}\\
=&O\Big(\int {\var_{\ep}^{2}}\big(\sum_{j=1}^{k}U_{\ep,y_{\ep,j}}\var_{\ep}\big)^{-1}\Big)
=O\Big(\|\var_{\ep}\|_{*}\|\var_{\ep}\|^{2}_{\ep}\Big)=o(1)\|\var_{\ep}\|_{\ep}^{2}.
\end{aligned}\end{equation}
Combining \eqref{asymptotic behavior 2.1}, \eqref{property}-\eqref{5.12}, we get
\begin{equation*}
\|\var_{\ep}\|_{\ep}=\ep^{\frac{N}{2}}O(\ep^{2}+\ep|y_{\ep,j}-\xi_{j}|)=O(\ep^{\frac{N}{2}+2}).
\end{equation*}
\end{proof}

Now we devoted to prove Theorem \ref{thm: uniqueness}.
We argue by way of contradiction.
Assume $u_{\ep}^{(i)}=\displaystyle\sum_{j=1}^{k}U_{\ep,y^{(i)}_{\ep,j}}+\var_{\ep}^{(i)}(i=1,2)$
are two distinct solutions concentrating around $\xi_{j}$.
 Set
\[
\eta_{\ep}=\frac{u_{\ep}^{(1)}-u_{\ep}^{(2)}}{\|u_{\ep}^{(1)}-u_{\ep}^{(2)}\|_{L^{\wq}(\R^{N})}},
\]
then
\begin{equation}\label{eq: comparison eq. 1}
 -\ep^{2}\De\eta_{\ep}+V(x)\eta_{\ep}=C_{\ep}(x)\eta_{\ep},
\end{equation}
where
\[
C_{\ep}(x)= 2\Big[\log \big(u_{\ep}^{(1)}+t(u_{\ep}^{(2)}-u_{\ep}^{(1)})\big)+1\Big],\ \ \ 0\leq t\leq1.
\]
It is clear that
$\|\eta_{\ep}\|_{L^{\wq}(\R^{N})}=1$.
We will prove that
\begin{equation}\label{3.1}
 \|\eta_{\ep}\|_{L^{\wq}(\R^{N})}=o(1)
\end{equation} to obtain a contradiction.
For fixed $j\in\{1,\cdots,k\}$, set \[
\eta_{\ep,j}(x)=\eta_{\ep}(\ep x+y^{(1)}_{\ep,j}).
\]
To prove (\ref{3.1}), we will prove that
$ \|\eta_{\ep,j}\|_{L^{\wq}(B_{R}(0))}=o(1)$ and $\|\eta_{\ep,j}\|_{L^{\wq}(\R^{N}\backslash B_{R}(0))}=o(1)$ holds separately.

First we study the asymptotic behavior of $\eta_{\ep,j}$.
\begin{proposition}\label{prop: convergence of xi_epsilon}
There exist $d_{\beta,j}\in\R$, $\beta=1,\cdots,N,\ j=1,\cdots,k$, such that (up to a subsequence)
\begin{eqnarray*}
\eta_{\ep,j}\to\sum_{\beta=1}^{N}d_{\beta,j}\frac{\partial U^{j}}{\partial x_{\beta}},  & \text{in }C_{\loc}^{1}(\R^{N}),
\end{eqnarray*}
as $\ep\to0$, where $ U^{j}$ solves
\[
-\ep^{2}\De U^{j}+V(\xi_{j})U^{j}=U^{j}\log (U^{j})^{2}.
\] .\end{proposition}
\begin{proof}
We will prove that the limiting function of $\eta_{\ep,j}$ belongs to the kernel of the linear operator associated to $U^{j}$.

In view of $\|\eta_{\ep,j}\|_{L^{\wq}(\R^{N})}\leq1$, the elliptic regularity theory implies
that $\eta_{\ep,j}\in C_{\loc}^{1,\theta}(\R^{N})$  with respect
to $\ep$ for some $\theta\in(0,1)$.
As a consequence, we assume (up to a subsequence) that
\begin{eqnarray*}
\eta_{\ep,j}\to\eta_{j} &  & \text{in }C_{\loc}^{1}(\R^{N}).
\end{eqnarray*}
We claim that $\eta_{j}$ satisfies
\begin{equation}
-\De\eta_{j}+V(\xi_{j})\eta_{j}=2\Big[\log U^{j}+1\Big]\eta_{j}.\label{eq: blow-up equations}
\end{equation}
Then by the fact that that $U^{j}$ is nondegenerate, we have $\eta_{j}=\sum_{\beta=1}^{N}d_{\beta,j}\frac{\partial U^{j}}{\partial x_{\beta}}$ for some $d_{\beta,j}\in\R$ ($\beta=1,\cdots,N$), and thus
Proposition \ref{prop: convergence of xi_epsilon} is proved.

Next, we prove (\ref{eq: blow-up equations}).
From \eqref{eq: comparison eq. 1}, we have $\eta_{\ep,j}$ satisfies
\begin{equation}\label{origion eq}
-\De\eta_{\ep,j}=-\ep^{2}\De\eta_{\ep}(\ep x+y^{(1)}_{\ep,j})=-V(\ep x+y^{(1)}_{\ep,j})\eta_{\ep,j}+C_{\ep}(\ep x+y^{(1)}_{\ep,j})\eta_{\ep,j}.
\end{equation}
Now we estimate $C_{\ep}(\ep x+y^{(1)}_{\ep,j})$.
From \eqref{asymptotic behavior 2.1},
\begin{equation}\label{minius}
\begin{aligned}
  U_{\ep,y^{(1)}_{\ep,t}}-U_{\ep,y^{(2)}_{\ep,t}}
  &=\frac{y^{(1)}_{\ep,t}-y^{(2)}_{\ep,t}}{\ep}
  \nabla U_{y_{t}}\Big(\frac{x-y^{(1)}_{\ep,t}+\theta(y^{(1)}_{\ep,t}-y^{(2)}_{\ep,t})}{\ep}\Big)\\
  &=o(1)\nabla U_{y_{t}}\Big(\frac{x-y^{(1)}_{\ep,t}+\theta(y^{(1)}_{\ep,t}-y^{(2)}_{\ep,t})}{\ep}\Big),
  \end{aligned}
\end{equation}
where $0<\theta<1,\ t=1,\cdots,k$ and $U_{y_{t}}$ satisfies
\[
-\De U_{y_{t}}+V(y_{\ep,t})U_{y_{t}}=U_{y_{t}}\log U_{y_{t}}^{2}.
\]
For simplicity, here and what follows, we denote \[
 z_{\ep,t}:=\frac{x-y^{(1)}_{\ep,t}+\theta(y^{(1)}_{\ep,t}-y^{(2)}_{\ep,t})}{\ep}.
 \]
Then,
\begin{equation}\label{big u minus}
\begin{aligned}
u_{\ep}^{(1)}-u_{\ep}^{(2)}&=\sum_{t=1}^{k}\Big(U_{\ep,y^{(1)}_{\ep,t}}-U_{\ep,y^{(2)}_{\ep,t}}\Big)
+O(|\var_{\ep}^{(1)}|+|\var_{\ep}^{(2)}|)\\
&=o(1)\sum_{t=1}^{k}\nabla U_{y_{t}}\Big(z_{\ep,t}\Big)+O(|\var_{\ep}^{(1)}|+|\var_{\ep}^{(2)}|).
 \end{aligned}
 \end{equation}
 So, for $x\in B_{d}(y^{(1)}_{\ep,j})$
 \[\begin{aligned}
C_{\ep}(x)=& 2\Big[\log \big(u_{\ep}^{(1)}+t(u_{\ep}^{(2)}-u_{\ep}^{(1)})\big)+1\Big]\\
=& 2\log\Big(\sum_{s=1}^{k}U_{\ep,y^{(1)}_{\ep,s}}(x)
+o(1)\sum_{t=1}^{k}\nabla U_{y_{t}}(z_{\ep,t})+O(|\var_{\ep}^{(1)}|+|\var_{\ep}^{(2)}|)\Big)+2,
 \end{aligned}\]
Then, we know
 \[\begin{aligned}
C_{\ep}(\ep x+y^{(1)}_{\ep,j})
=& 2\log\Big(U_{\ep,y^{(1)}_{\ep,j}}(\ep x+y^{(1)}_{\ep,j})+o(1)\nabla U_{y_{j}}(z_{\ep,j,j})
+\sum_{s\neq j}U_{\ep,y^{(1)}_{\ep,s}}(\ep x+y^{(1)}_{\ep,j}) \\
& +o(1)\sum_{t\neq j}\nabla U_{y_{t}}(z_{\ep,t,j})+O\big(|\var_{\ep}^{(1)}(\ep x+y^{(1)}_{\ep,j})|+|\var_{\ep}^{(2)}(\ep x+y^{(1)}_{\ep,j})|\big)\Big)+2\\
=& 2\log\Big(U_{\ep,y^{(1)}_{\ep,j}}(\ep x+y^{(1)}_{\ep,j})+o(1)\nabla U_{y_{j}}(z_{\ep,j,j})+o(1)\Big)+2\\
&+O\Big(\frac{\sum_{s\neq j}U_{\ep,y^{(1)}_{\ep,s}}(\ep x+y^{(1)}_{\ep,j})+o(1)\sum_{t\neq j}\nabla U_{y_{t}}(z_{\ep,t,j})}{U_{\ep,y^{(1)}_{\ep,j}}(\ep x+y^{(1)}_{\ep,j})}\Big)\\
=& 2\log\Big(U_{\ep,y^{(1)}_{\ep,j}}(\ep x+y^{(1)}_{\ep,j})+o(1)\nabla U_{y_{j}}(z_{\ep,j,j})\Big)+2+O(e^{-\frac{\gamma}{\ep}}),
\ \ x\in B_{\frac{d}{\ep}}(0),
 \end{aligned}\]
where $$z_{\ep,t,j}=\frac{\ep x+y^{(1)}_{\ep,j}-y^{(1)}_{\ep,t}+\theta(y^{(1)}_{\ep,t}-y^{(2)}_{\ep,t})}{\ep}$$ and
$\gamma>0$ is a constant.
Now recall \eqref{origion eq}, we know
\begin{equation}\label{origion eq2}
\begin{aligned}
-&\De\eta_{\ep,j}+V(\ep x+y^{(1)}_{\ep,j})\eta_{\ep,j}\\&=\Big(2\log\big(U_{\ep,y^{(1)}_{\ep,j}}(\ep x+y^{(1)}_{\ep,j})+o(1)\nabla U_{y_{j}}(z_{\ep,j})+o(1)\big)+2+O(e^{-\frac{\gamma}{\ep}})\Big)\eta_{\ep,j}.
\end{aligned}\end{equation}
Letting $\ep\rightarrow0$ in \eqref{origion eq2},
we obtain (\ref{eq: blow-up equations}). The proof is completed.
\end{proof}

 Next, similar to Lemma \ref{estimate about u}, we find
\begin{lemma}\label{minus lemma}
  There exists a small constant $d>0$, such that
  \[
  |\eta_{\ep}(x)|+
  |\nabla\eta_{\ep}(x)|=O\big(e^{-\frac{d}{\ep}}\big),\ \ \forall x\in\R^{N}\setminus\bigcup_{j=1}^{k}B_{d}(y^{(1)}_{\ep,j}).
  \]
\end{lemma}
\begin{proposition} \label{lem: vanishing limit function}Let $d_{\beta,j}$
be defined as in Proposition \ref{prop: convergence of xi_epsilon}.
Then
\begin{eqnarray*}
d_{\beta,j}=0 &  & \text{for }\beta=1,\cdots,N.\  \ j=1,\cdots,k.
\end{eqnarray*}
 \end{proposition}
 \begin{proof}


Applying (\ref{eq: Pohozaev}) to $u_{\ep}^{(1)}$ and $u_{\ep}^{(2)}$
with $\Om=B_{d}(y^{(1)}_{\ep,j})$, where $d$ is chosen such that $0<d<\min_{i\neq j}|y^{(1)}_{\ep,i}-y^{(1)}_{\ep,j}|$, we have
\begin{equation}\label{3.10}
\begin{aligned} \int_{B_{d}(y^{(1)}_{\ep,j})}&\frac{\pa V(x)}{\pa x_{i}}(u_{\ep}^{(1)}+u_{\ep}^{(2)})\eta_{\ep}\\
=&\int_{\pa B_{d}(y^{(1)}_{\ep,j})}\left(\ep^{2}\langle\nabla(u_{\ep}^{(1)}+u_{\ep}^{(2)}),\nabla\eta_{\ep}\rangle
 +V(x)\langle u_{\ep}^{(1)}+u_{\ep}^{(2)}, \eta_{\ep}\rangle\right)\nu_{i}\\
 &-2\ep^{2}\int_{\pa B_{d}(y^{(1)}_{\ep,j})}\left(\frac{\pa \eta_{\ep}}{\pa \nu}\frac{\pa u_{\ep}^{(1)}}{\pa x_{i}}+\frac{\pa \eta_{\ep}}{\pa x_{i}}\frac{\pa u_{\ep}^{(1)}}{\pa \nu}\right)
    -2\int_{\pa B_{d}(y^{(1)}_{\ep,j})}A_{\ep}(x)\eta_{\ep}(x)\nu_{i},
 \end{aligned}
\end{equation}
 where $1\leq i\leq N$ and
 \[
 A_{\varepsilon}(x)=\frac{(u_{\ep}^{(1)})^{2}\log (u_{\ep}^{(1)})^{2}-(u_{\ep}^{(2)})^{2}\log (u_{\ep}^{(2)})^{2}}{u_{\ep}^{(1)}-u_{\ep}^{(2)}}=4\tilde{u}_{\ep}\log \tilde{u}_{\ep}+2\tilde{u}_{\ep}
 \]
with $\tilde{u}_{\ep}=u_{\ep}^{(1)}+\theta(u_{\ep}^{(1)}-u_{\ep}^{(2)}).$
 By \eqref{big u minus}, we have for $x\in\pa B_{d}(y^{(1)}_{\ep,j})$,
 \[
 \tilde{u}_{\ep}=\sum_{s=1}^{k}U_{\ep,y^{(1)}_{\ep,s}}(x)
+o(1)\sum_{t=1}^{k}\nabla U_{y_{t}}(z_{\ep})+O(|\var_{\ep}^{(1)}|+|\var_{\ep}^{(2)}|)=O(e^{-\frac{\gamma}{\ep}}),
 \]
 Notice that $|\eta_{\ep}|\leq1$, so
 \[
 \int_{\pa B_{d}(y^{(1)}_{\ep,j})}A_{\ep}(x)\eta_{\ep}(x)\nu_{i}=O(e^{-\frac{\gamma}{\ep}}).
 \]
By \eqref{boundary estimate} and Lemma \ref{minus lemma}, we have
 \[
 \int_{\pa B_{d}(y^{(1)}_{\ep,j})}\left(\ep^{2}\langle\nabla(u_{\ep}^{(1)}+u_{\ep}^{(2)}),\nabla\eta_{\ep}\rangle
 +V(x)\langle u_{\ep}^{(1)}+u_{\ep}^{(2)}, \eta_{\ep}\rangle\right)\nu_{i}=O(e^{-\frac{\gamma}{\ep}})
 \]
 and
 \[
 2\ep^{2}\int_{\pa B_{d}(y^{(1)}_{\ep,j})}\left(\frac{\pa \eta_{\ep}}{\pa \nu}\frac{\pa u_{\ep}^{(1)}}{\pa x_{i}}+\frac{\pa \eta_{\ep}}{\pa x_{i}}\frac{\pa u_{\ep}^{(1)}}{\pa \nu}\right)=O(e^{-\frac{\gamma}{\ep}}).
 \]
 So, \eqref{3.10} equivalent to
 \begin{equation}\label{3.101}
   \int_{B_{d}(y^{(1)}_{\ep,j})}\frac{\pa V(x)}{\pa x_{i}}(u_{\ep}^{(1)}+u_{\ep}^{(2)})\eta_{\ep}
   =O(e^{-\frac{\gamma}{\ep}}).
 \end{equation}
As $V(x)$ satisfies $(V_{2})$, for $l=1,\cdots,N$, we have
\begin{equation}\label{3.102}
\begin{aligned}
\int_{B_{d}(y^{(1)}_{\ep,j})}&\frac{\pa V(x)}{\pa x_{i}}(u_{\ep}^{(1)}+u_{\ep}^{(2)})\eta_{\ep}\\=&\int_{B_{d}(y^{(1)}_{\ep,j})}\Big[\big\langle\frac{\nabla^{2} V(\xi_{j})}{\pa x_{i}\pa x_{l}},x_{l}-\xi_{j,l}\big\rangle +O\Big(|x-\xi_{j}|^{2}\Big)\Big] (u_{\ep}^{(1)}+u_{\ep}^{(2)})\eta_{\ep}.  \end{aligned}
\end{equation}
From \eqref{big u minus}, we have
\[
u_{\ep}^{(1)}+u_{\ep}^{(2)}=2\sum_{s=1}^{k}U_{\ep,y^{(1)}_{\ep,s}}(x)
+o(1)\sum_{t=1}^{k}\nabla U_{y_{t}}(z_{\ep,t})+O(|\var_{\ep}^{(1)}|+|\var_{\ep}^{(2)}|).
\]
Then, we find
\[
\begin{aligned}
 \int_{B_{d}(y^{(1)}_{\ep,j})}&\langle\frac{\nabla^{2} V(\xi_{j})}{\pa x_{i}\pa x_{l}},x_{l}-\xi_{j,l}\rangle (u_{\ep}^{(1)}+u_{\ep}^{(2)})\eta_{\ep}\\
  =&2\frac{\nabla^{2} V(\xi_{j})}{\pa x_{i}\pa x_{l}}\int_{B_{d}(y^{(1)}_{\ep,j})}(x_{l}-\xi_{j,l})U_{\ep,y^{(1)}_{\ep,j}}(x)\eta_{\ep}\\
  &+o(1)\frac{\nabla^{2} V(\xi_{j})}{\pa x_{i}\pa x_{l}}\int_{B_{d}(y^{(1)}_{\ep,j})}(x_{l}-\xi_{j,l})\nabla U_{y_{j}}(z_{\ep,j})\eta_{\ep}\\
  &+O\Big(\int_{B_{d}(y^{(1)}_{\ep,j})}(x_{l}-\xi_{j,l})\big(|\var_{\ep}^{(1)}|+|\var_{\ep}^{(2)}|\big)
  \eta_{\ep}\Big)
  +O(e^{-\frac{\gamma}{\ep}}).\end{aligned}
\]
By \eqref{asymptotic behavior 2.1} and Proposition \ref{prop: convergence of xi_epsilon}, we have
\[
\begin{aligned}
\int_{B_{d}(y^{(1)}_{\ep,j})}&(x_{l}-\xi_{j,l})U_{\ep,y^{(1)}_{\ep,j}}(x)\eta_{\ep}\\
  =&\ep^{N+1}\int_{B_{\frac{d}{\ep}}(0)}(z_{l}-\frac{y^{(1)}_{\ep,j,l}-\xi_{j,l}}{\ep})
      U_{\ep,y^{(1)}_{\ep,j}}(\ep z+y^{(1)}_{\ep,j})\Big(\sum_{\beta=1}^{N}d_{\beta,j}\frac{\partial U^{j}}{\partial z_{\beta}}+o(1)\Big)\\
   =&\ep^{N+1}d_{l,j}\int_{B_{\frac{d}{\ep}}(0)}z_{l}U^{j}\frac{\partial U^{j}}{\partial z_{l}}+o(\ep^{N+1}).
  \end{aligned}
\]
Similarly,  as $\|\eta_{\ep}\|=O(\varepsilon^\frac{N}{2})$ and $\|\var_{\ep}\|=O(\varepsilon^{\frac{N}{2}+2})$, we have
\[
\int_{B_{d}(y^{(1)}_{\ep,j})}|x_{l}-\xi_{j,l}|\Big(o(1)|\nabla U_{y_{j}}(z_{\ep,j})|+|\var_{\ep}^{(1)}|+|\var_{\ep}^{(2)}|\Big)\eta_{\ep}=o(\ep^{N+1}),
\]
and
\[
\int_{B_{d}(y^{(1)}_{\ep,j})}|x-\xi_{j}|^{2}(u_{\ep}^{(1)}+u_{\ep}^{(2)})\eta_{\ep}=O(\ep^{N+2}).
\]
Combining above with \eqref{3.101}, \eqref{3.102}, we have
\[
2d_{l,j}\int_{B_{\frac{d}{\ep}}(0)}z_{l}U^{j}\frac{\partial U^{j}}{\partial z_{l}}=o(1).
  \]
Then $d_{l,j}=0$ for $l=1,\cdots,N,\ j=1,\cdots,k,$ since $U^{j}$ is a radially symmetric decreasing function.
 \end{proof}

\begin{proof}[\textbf{Proof of Theorem \ref{thm: uniqueness}}]
  Propositions \ref{prop: convergence of xi_epsilon} and   \ref{lem: vanishing limit function} show that
   \[
   |\eta_{\ep,j}|=o(1),\ \ \ x\in B_{R}(0),
   \] for any $j=1,\cdots,k$, which means \[
    |\eta_{\ep}|=o(1),\ \ \ x\in B_{R\ep}(y^{(1)}_{\ep,j}).
   \]
   On the other hand, by using maximum principle, we can prove \[
|\eta_{\ep}|=o(1),\ \ \ x\in \R^{N}\backslash \bigcup_{j=1}^{k}B_{R\ep}(y^{(1)}_{\ep,j}).
 \]
 we can refer to \cite[Proposition 3.5]{Cao-Li-Luo-2015} for the similar detail proof. Consequently, we get (\ref{3.1}), which contradict to  $\|\eta_{\ep}\|_{\infty}=1.$
  The proof of local uniqueness is completed.
\end{proof}

\noindent\textbf{Acknowledgments}

The authors would like to thank Profs. Shuangjie Peng and Shusen Yan for many useful suggestions during
the preparation of this paper. The first author (Peng Luo) was partially supported by NSFC grants (No.11701204, No.11831009)

\end{document}